\documentclass[12pt]{article}
\usepackage{cite}
\usepackage{amsmath,amssymb,amsthm,mathrsfs}
\usepackage{titlesec,hyperref}
\usepackage{color}
\usepackage{fancyhdr}
\usepackage{cases}
\usepackage[margin=2.5cm]{geometry}
\usepackage{enumerate}
\usepackage[integrals]{wasysym}
\usepackage{bm}
\usepackage{graphicx}
\usepackage{caption}
\usepackage{float}
\usepackage{subfigure}
\usepackage{indentfirst}
\newcommand{\dd}{\mathop{}\!\mathrm{d}}
\let\del\partial

\newcommand{\MM}{\mathring M}
\newcommand{\MMM}{\mathring{\bar M}}
\newcommand{\RR}{\mathring R}
\newcommand{\R}{\mathbb R}
\newcommand{\ZZ}{\mathbb Z}
\newcommand{\RRR}{\mathring{\bar R}}
\newcommand{\ootimes}{\mathbin{\mathring{\otimes}}}

\let\div\relax
\DeclareMathOperator{\div}{div}

\newcommand{\vin}[0]{v^{\textup{in}}}
\newcommand{\bin}[0]{b^{\textup{in}}}
\newcommand{\pex}[0]{\textsf{\textit{p}}}
\newcommand{\vex}[0]{\textsf{\textit{v}}}
\newcommand{\vv}[0]{\bar v}
\newcommand{\bb}[0]{\bar b}

\newcommand{\ppp}[0]{\bar{\bar p}}
\newcommand{\pp}[0]{\bar p}

\newcommand{\PP}[0]{\mathbb{P}_{\neq 0}}

\newcommand{\bex}[0]{\textsf{\textit{b}}}

\newcommand{\Rosc}{R_{q+1}^\textup{osc}}

\newcommand{\Mosc}{M_{q+1}^\textup{osc}}
\newcommand{\Mlin}{M_{q+1}^\textup{lin}}
\newcommand{\Rlin}{R_{q+1}^\textup{lin}}
\newcommand{\TT}[0]{\mathsf{T}}
\newcommand{\TTT}[0]{\mathbb{T}}
\newcommand{\e}{\epsilon}
\def\dashint{\,\ThisStyle{\ensurestackMath{%
  \stackinset{c}{.2\LMpt}{c}{.5\LMpt}{\SavedStyle-}{\SavedStyle\phantom{\int}}}%
  \setbox0=\hbox{$\SavedStyle\int\,$}\kern-\wd0}\int}
\def\ddashint{\,\ThisStyle{\ensurestackMath{%
  \stackinset{c}{.2\LMpt}{c}{.5\LMpt+.2\LMex}{\SavedStyle-}{%
    \stackinset{c}{.2\LMpt}{c}{.5\LMpt-.2\LMex}{\SavedStyle-}{%
      \SavedStyle\phantom{\int}}}}\setbox0=\hbox{$\SavedStyle\int\,$}\kern-\wd0}\int}
\usepackage{scalerel,stackengine}
\stackMath
\newcommand\widecheck[1]{%
\savestack{\tmpbox}{\stretchto{%
  \scaleto{%
    \scalerel*[\widthof{\ensuremath{#1}}]{\kern-.6pt\bigwedge\kern-.6pt}%
    {\rule[-\textheight/2]{1ex}{\textheight}}
  }{\textheight}%
}{0.5ex}}%
\stackon[1pt]{#1}{\scalebox{-1}{\tmpbox}}%
}
\newcommand{\p}{\phi_{(\gamma,\frac{1}{2},k)}}
\newcommand{\pd}{\phi_{(\gamma,\frac{1}{2},k')}}
\newcommand{\g}{g_{(2,\sigma)}}
\newcommand{\coloneq}{\mathrel{\mathop:}=}

\DeclareMathOperator{\curl}{curl}

\newtheorem{thm}{Theorem}[section]
\newtheorem{cor}[thm]{Corollary}

\newtheorem{lem}[thm]{Lemma}
\newtheorem{prop}[thm]{Proposition}
\theoremstyle{definition}
\newtheorem{defn}[thm]{Definition}

\theoremstyle{remark}
\newtheorem{rem}[thm]{Remark}
\numberwithin{equation}{section}

\allowdisplaybreaks

\begin{document}
\title{Sharp and strong non-uniqueness for the magneto-hydrodynamic equations}

\author{
Yao $\mbox{Nie}^{1}$ \footnote{Institute  of Applied Physics and Computational Mathematics, email: nieyao930930@163.com}
\quad and\quad
Weikui $\mbox{Ye}^{1}$ \footnote{Institute  of Applied Physics and Computational Mathematics, email:
904817751@qq.com}
}
\date{}
\maketitle
\begin{abstract}
In this paper, we prove a sharp and strong non-uniqueness for a class of weak solutions to the three-dimensional  magneto-hydrodynamic (MHD) system. More precisely, we show that any weak solution $(v,b)\in L^p_tL^{\infty}_x$ is non-unique in $L^p_tL^{\infty}_x$ with $1\le p<2$, which reveals the strong non-uniqueness, and the sharpness in terms of the classical Ladyzhenskaya-Prodi-Serrin criteria at endpoint $(2, \infty)$. Moreover, for any $\epsilon>0$, we can also construct infinitely many  weak solutions of the ideal MHD system in $L^1_tC^{1-\epsilon}_x$. Our result shows the non-uniqueness for any weak solution $(v,b)$ including non-trivial magnetic field $b$.
\end{abstract}

\section{Introduction}
In many natural phenomena, magnetic fields are closely related to a large number of natural and man-made fluids. The model of magneto-hydrodynamics (MHD) system is already widely used to  study the interaction between magnetic fields and moving, conducting flows.  In this paper, we consider the MHD equations on the torus $\TTT^3=\R^3/\ZZ^3$:
\begin{equation}
\left\{ \begin{alignedat}{-1}
&\del_t v-\Delta v+ \div (v\otimes v)  +\nabla p   =\div (b\otimes b) ,
 \\
 &\del_t b-\Delta b+\div (v\otimes b)     = \div (b\otimes v) ,
 \\
 & \div v = 0,~\div b= 0,
\end{alignedat}\right.  \label{mhd}
\end{equation}
associated with mean-free initial data $(v,b) |_{t=0}=(v_0, b_0)$, where $v(x,t)$, $b(t,x)$  and $p(x,t)$ represent the fluid velocity,  the magnetic field  and the scalar pressure,  respectively.

When $b=0$, system \eqref{mhd} is reduced to the classical incompressible Navier-Stokes equations. As is known, Leray \cite{Le34} proved global existence  of weak solutions for  the Navier-Stokes equations  with initial data $v_0\in L^2(  \R^3)$. And this result was further developed by Hopf \cite{Ho51} in general domains. We call these solutions introduced in \cite{Ho51, Le34} as Leray-Hopf weak solutions. Although the existence theory of Leray-Hopf weak solutions has been established for nearly one hundred years,  the uniqueness and regularity of the weak solutions for the 3D Navier-Stokes equations  remain open. Therefore, a large number of mathematical
researchers are devoted to seeking the uniqueness and regularity criteria to show how far we are from resolving
this challenging problem. One of the classical regularity criteria is the Ladyzhenskaya-Prodi-Serrin criterion, which reveals that a Leray-Hopf weak solution $v$ of the 3D Navier-Stokes equations is  regular,  provided
$v\in\mathbf{X}^{p,q}$ satisfies  \cite{ESS, GKP, Serrin}:
\begin{equation}
\frac{2}{p}+\frac{3}{q}\leq1,~~~\mathbf{X}^{p,q}:=\left\{ \begin{alignedat}{-1}
L^p([0,T];L^q(\Omega)),~~p\neq\infty,
 \\
 C([0,T];L^q(\Omega)),~~p=\infty.
\end{alignedat}\right.  \label{serrinzhunze}
\end{equation}
 where $\Omega$ can be the whole space $\R^3$, a bounded domain or a periodic area $\TTT^3$. As a matter of fact, the Ladyzhenskaya-Prodi-Serrin criterion can also serve as a sufficient condition for the uniqueness of $L^2_{t,x}([0, T]\times\R^3)$ weak solution. We refer readers to \cite{FJR, LionsMas, FLT} and references therein. For periodic domain $\TTT^3$, Cheskidov and Luo \cite{1Cheskidov} proved that any $L^2_{t,x}$ weak solution in $X^{p,q}$ with $\tfrac{2}{p}+\frac{3}{q}\leq1$ of the Navier-Stokes equations is a Leray-Hopf weak solution  and regular (also unique). This positive result naturally makes one consider whether a $L^2_{t,x}$ weak solution is unique or not in $X^{p,q}$ with $\frac{2}{p}+\frac{3}{q}>1 $. Recently, Cheskidov-Luo \cite{1Cheskidov} answered this question and their result shows the sharpness of the Ladyzhenskaya-Prodi-Serrin criteria $\frac{2}{p}+\frac{3}{q}\le 1$ at the endpoint $(p, q)=(2,\infty)$ in the periodic setting.

The same as for the 3D Navier-Stokes equations, the regularity and uniqueness of the Leray-Hopf weak solution for the 3D MHD equations remain unsolved. The existence of Leray-Hopf weak solution of system \eqref{mhd} was proved by Wu \cite{67Wu}. The definition of Leray-Hopf weak solution of system \eqref{mhd} on the periodic domain $\TTT^3$ is as follows:
\begin{defn}
Let $(v_0, b_0)\in L^2(\TTT^3)$\footnote{Here and throughout the paper, we denote $(f , g)\in X\times X$ by $(f, g)\in X$ and $\|f\|^{n}_{X}+\|g\|^{n}_{X}$ by $\|(f,g)\|^{n}_{X}$.}  with zero mean and $\div v_0=\div b_0=0$.   $(v, b, p)$ is called a \emph{Leray-Hopf}  weak solution of system \eqref{mhd} if
\begin{itemize}
  \item [$\bullet$]$(v, b, p)$ satisfies equations \eqref{mhd} in the distribution sense. Moreover,
      \[\big(v(t), b(t)\big)\to(v_0, b_0) \,\,\text{weakly}\,\,\text{in}\,\,L^2(\TTT^3)\,\,\text{as}\,t\to 0.\]
  \item [$\bullet$] $(v, b)\in L^\infty_{\text{loc}}\big(\R^+; L^2(\TTT^3)\big)\cap L^2_{\text{loc}}\big(\R^+; H^1(\TTT^3)\big)$.
  \item [$\bullet$] For a.e. $t>0$,
  \begin{align*}
&\|(v(t), b(t))\|^2_{L^2(\TTT^3)}+2\int_0^t\|(\nabla v( s),\nabla b(s))\|^2_{L^2(\TTT^3)}\dd s
\le\|(v_0, b_0)\|^2_{L^2(\TTT^3)}.
  \end{align*}
\end{itemize}
\end{defn}
With respect to MHD equations, there have been established some Ladyzhenskaya-Prodi-Serrin-type criteria, e.g. \cite{serrin0,serrin1,serrin2,serrin3,serrin4}. One can claim that the uniqueness for the Leray-Hopf weak solution $(v,b)$ of system \eqref{mhd} if $(v,b)\in\mathbf{X}^{p,q}$ with $\frac{2}{p}+\frac{3}{q}\le 1$ on the whole space $\R^3$. Actually, for a $ L^2_{t,x}$ weak solution $(v,b)$ in the periodic setting, one can also obtain the uniqueness result provided that $(v, b)$ satisfies  Ladyzhenskaya-Prodi-Serrin criteria \eqref{serrinzhunze}. Here a $ L^2_{t,x}$ weak solution $(v, b)$ of  system \eqref{mhd} is defined as follows:
\begin{defn}\label{def} Let $(v_0, b_0)\in L^2(\TTT^3)$ be divergence-free in the sense of distributions and have zero mean\footnote{Throughout this paper, ``zero mean'' is ``zero spatial mean''.}. We say that  $(v,b)\in L^2([0, T]\times\TTT^3)$ is a \emph{weak solution}  to the MHD equations~\eqref{mhd} if \begin{itemize}
     \item [(1)] For a.e. $t\in [0,T]$, $(v(\cdot, t), b(\cdot, t))$ is divergence-free in the sense of distributions;
     \item [(2)]For all divergence-free test functions $\phi\in C^\infty_0([0,T)\times \mathbb T^3)$,
     \begin{align}\nonumber
\int_0^T \int_{\mathbb T^3} (\del_t-\Delta)\,\phi v+\nabla \phi : (v\otimes v-b\otimes b) \dd x \dd t=  -\int_{\mathbb T^3} v_0 \phi(0,x)\dd x,\\
\int_0^T \int_{\mathbb T^3} (\del_t-\Delta)\,\phi b+\nabla \phi : (v\otimes b-b\otimes v) \dd x \dd t=-\int_{\mathbb T^3}b_0\phi(0,x)\dd x.\nonumber
\end{align}
   \end{itemize}
\end{defn}
Throughout this paper, we call $(v,b)$ a \emph{weak solution} in the sense of Definition \ref{def}. Taking advantage of the proof of Theorem \ref{Serrin} in \cite{1Cheskidov}, one can infer the following result:
\begin{prop}\label{Serrin}
Let $(v,b)$ be a weak solution of the MHD system \eqref{mhd}. If $(v,b)\in X^{p,q}$ with $\frac{2}{p}+\frac{3}{q}\le 1$,  then
\begin{itemize}
  \item [(1)]$(v,b)$ is unique in the class of $X^{p,q}$ weak solutions,
  \item [(2)]$(v,b)$ is a Leray-Hopf solution, and regular on $(0, T]$.
\end{itemize}
\end{prop}
The results in \cite{1Cheskidov} imply the non-uniqueness of weak solutions of the MHD system \eqref{mhd} in the class $L^p_tL^\infty_x(1\le p<2)$ with the trivial magnetic field $b=0$ on periodic domain $\TTT^3$. One natural question is:
\vskip 2mm
\emph{For $1\le p<2$, is a weak solution $(v, b)\in L^p([0,T]; L^\infty(\TTT^3))$ of system \eqref{mhd} unique in $ L^p([0,T]; L^\infty(\TTT^3))$  for the non-trivial magnetic field $b$?}
\vskip 2mm
Inspired by \cite{1Cheskidov, 2018Wild, 2Beekie}, we aim to solve the above problem and obtain the non-uniqueness results via  convex integration scheme. In recent years, the  non-uniqueness problems of weak solutions to hydrodynamic
models have caught researchers' interest and some progress has been made on related issues. For example, with respect to the Euler equations, De Lellis and Sz\'{e}kelyhidi  in the pioneering paper \cite{27DeLellis} introduced the convex integration scheme
and constructed the non-unique weak solutions in $L^{\infty}_{t,x}$ with space-time compact support, see also \cite{28DeLellis}. The Onsager conjecture was finally solved in $
C^{\beta}_{x,t}~ (0 < \beta < 1/3)$ by Isett \cite{Isett}. We refer readers to \cite{zbMATH07038033, zbMATH06312794, zbMATH06710292, zbMATH07370998, Rosa2021DimensionOT} and and references therein for more results on non-uniqueness of weak solutions to the Euler equations. For the Navier-Stokes  equations, there were relatively fewer results on the non-uniqueness problems. Buckmaster and Vicol in \cite{13Nonuniqueness} made the first important break-through via a ${L^2_x}$-based intermittent convex integration scheme. Subsequently, Buckmaster, Colombo and  Vicol \cite{2018Wild} proved that the wild solutions can be generated by  $H^3$ initial data.  Cheskidov and Luo \cite{1Cheskidov} exhibited the sharpness of the Ladyzhenskaya-Prodi-Serrin criteria $\frac{2}{p}+\frac{d}{q}\le 1$ at the endpoint $(p,q)=(2,\infty)$.  Albritton, Bru\'{e} and Colombo \cite{leray} proved the non-uniqueness of the Leray-Hopf weak solutions with a special force by skillfully constructing a ``background'' solution which is unstable for the Navier-Stokes dynamics in similarity variables. { In terms of the ideal MHD system}, Faraco, Lindberg and Sz\'{e}kelyhidi
\cite{37Faraco} showed that the  integrability condition $(v,b)\in L^3_{t,x}$ for the magnetic helicity conservation is sharp. In \cite{36Faraco},  Faraco, Lindberg and Sz\'{e}kelyhidi proved that there exist
infinitely many non-trivial weak solutions with compact support in space $L^{\infty}_{t,x}$.  Beekie, Buckmaster and Vicol \cite{2Beekie} constructed distributional solutions in $C_tL^2_x$ breaking the magnetic helicity conservation and essentially showed the non-uniqueness of weak solutions. Recently, there have some non-uniqueness  results of the 3D viscous and resistive MHD equations. For instance, in the framework of $H^{\epsilon}_{t,x}$ with $\epsilon$ sufficiently small, Li, Zeng and Zhang \cite{lyc} proved the non-uniqueness of a generalized MHD system.

Compared with the Navier-Stokes equations,  the anti-symmetric structure $\div(v\otimes b-b\otimes v)$ of \eqref{mhd} gives rise to the major difficulty to obtain the non-uniqueness in $L^p_tL^{\infty}_x$ with $1\le p<2$. More precisely, when dealing with the MHD system \eqref{mhd}, the ``Mikadow flows'' introduced in \cite{1Cheskidov} for the Navier-Stokes equations would not work since the multiple oscillation directions lead to extra errors in oscillation terms. Motivated by \cite{2Beekie,1Cheskidov}, we construct the perturbation by combining the shear intermittent flows with the temporal concentration functions on convex integration scheme, where the shear intermittent flows effectively eliminate the oscillation errors and  the temporal concentration functions help the intermittent shear flows balance the dissipative errors.

Before giving our results, we show two  non-uniqueness definitions introduced in\cite{1Cheskidov}:
\begin{align}
&\bullet~\text{``Weak non-uniqueness'': there exists a non-unique weak solution in the class X.}\notag\\
&\bullet~\text{``Strong non-uniqueness'': any weak solution in the class X is non-unique.}\notag
\end{align}
We can prove the non-uniqueness of weak solutions in a strong sense and our main results are as follows:
\begin{thm}[Sharp and strong non-uniqueness]\label{t:main0}
Let $1\le p<2$. Any weak solution $(v,b)$ of  \eqref{mhd}  in $L^p_tL^{\infty}_x$ is non-unique. Moreover, there exist non-Leray-Hopf weak solutions $(v,b)$ in $L^p_tL^{\infty}_x$.
\end{thm}
\begin{rem}Given $1\le p<2$, Cheskidov and Luo \cite{1Cheskidov} proved that a weak solution $v$ of the Navier-Stokes equations in  $ L^p(0,T; L^\infty(\TTT^3))$ is not unique\emph{ if $v$ has at least one interval regularity}. Our result shows the non-uniqueness of any weak solution $(v, b)$ without regularity assumption.
\end{rem}

We now give a main theorem of this paper,  which immediately shows Theorem \ref{t:main0}.
\begin{thm}\label{t:main}
Let $(v^{(1)}, b^{(1)})\in C^0([0, \widetilde{T}]; H^3(\TTT^3))$ and $(v^{(2)}, b^{(2)})\in C^0([0, \widetilde{T}]; H^3(\TTT^3)$ be two strong solutions of the MHD equations \eqref{mhd} with mean-free initial data $(v^{(1)}(0, x), b^{(1)}(0,x))$ and $(v^{(2)}(0, x), b^{(2)}(0, x))$, respectively. Fixed $ {T}\le\tfrac{\widetilde{T}}{4}$ and $1\le p<2$,
then there exists a weak solution $(v,b)$ of the Cauchy problem to \eqref{mhd} with initial data $(v,b)|_{t=0}=(v^{(1)}(0, x), b^{(1)}(0, x))$ satisfying
\[(v, b)\in L^p([0, \widetilde{T}]; L^\infty(\TTT^3))\cap
L^1([0, \widetilde{T}]; C^{1^-}(\TTT^3))\]
and
\begin{align*}
(v, b)\equiv(v^{(1)}, b^{(1)})\,\,\text{on}\,\,[0, 2 {T}],\quad \text{and}\quad(v, b)\equiv(v^{(2)}, b^{(2)})\,\,\text{on}\,\,[3 {T}, \widetilde{T}].
\end{align*}
\end{thm}
For the ideal MHD system, there has no dissipative effect preventing the nonlinear term from balancing the stress errors. Therefore, we can apply the methods of proving Theorem \ref{t:main0} and Theorem \ref{t:main} to show the following result:
\begin{cor}For the ideal MHD system, there exist infinitely many weak solutions $(v, b)\in L^1_t C^{1-}_x$ with the same smooth initial data.
\end{cor}
In the end of this section, we arise {some interesting questions} on the related  MHD systems:
\begin{itemize}
  \item [(1)]For non-resistive MHD equations, one could expect that the magnetic helicity is conserved for smooth solutions. Then  is  integrability condition $(v,b)\in L^3_{t,x}$ sharp for the magnetic helicity conservation?

  \item [(2)]{For the MHD system with magnetic resistivity but without dissipation}, does there exist solutions $(v, b)$ in $C^{\alpha}_{t,x}(\alpha<\frac{1}{3})$ with increasing energy? is weak solution $(v, b)$  unique  for $v\in L^1_x C^{1-}_x$?
\end{itemize}

\section{Outline of the proof}

As previously mentioned,  Theorem \ref{t:main0} can be immediately proved by Theorem \ref{t:main}. Therefore, we are concentrated on showing Theorem \ref{t:main} in this paper. The process of proving Theorem \ref{t:main} is based on a convex integration scheme inspired by \cite{2018Wild, 2Beekie, 1Cheskidov, 13Nonuniqueness}. Now, we give the sketch of the convex integration scheme.

For every integer $q\ge 1$, we investigate the  following relaxation  of \eqref{mhd} :
\begin{equation}
\left\{ \begin{alignedat}{-1}
&\del_t v_q-\Delta v_q+\div (v_q\otimes v_q)  +\nabla p_q   =\div (b_q\otimes b_q) + \div \RR_q,
 \\
 &\del_t b_q-\Delta b_q+\div (v_q\otimes b_q)     = \div (b_q\otimes v_q)+ \div \MM_q,
 \\
  &\nabla \cdot v_q = 0,~\nabla \cdot b_q = 0,
  \\ &(v_q,b_q) |_{t=0}=(\vin, \bin):=(v^{(1)}(0, x), b^{(1)}(0, x)),
\end{alignedat}\right.  \label{e:subsol-euler}
\end{equation}
where the \emph{Reynolds stress} $\RR_q$  and the \emph{magnetic  stress} $\MM_q$ are trace-free symmetric matrix and  skew-symmetric matrix, respectively. Here and below,  $v \otimes b\coloneq  (v_j b_i)_{i,j=1}^3$, and the divergence of a $3\times 3$ matrix $M=(M_{ij})_{i,j=1}^3$ is defined by $\div M$ with components $(\div M)_i \coloneq   \partial_j M_{ji}$.

Without loss of generality, we consider the non-uniqueness of weak solutions on time interval $[0,1]$. Assume that $(v^{(1)}, b^{(1)}), (v^{(2)}, b^{(2)})\in C^0([0, 1]; H^3(\TTT^3))$ are two strong solutions of the MHD equations \eqref{mhd} with mean-free initial data $(\vin, \bin)$ and $(v^{(2)}(0, x), b^{(2)}(0, x))$ respectively. Fixed $T>0$ such that
\[T\le\tfrac{1}{4},\]
our target is to construct a solution $(v_q, b_q, p_q,\RR_q, \MM_q)$ to \eqref{e:subsol-euler} such that
$\{(v_q, b_q)\}_{q\ge 1}$ is a Cauchy sequence in $L^2_{t,x}\cap L^p_tL^\infty_x\cap L^1_t C^{1-\e}_x$, the Reynolds and magnetic stresses $(\RR_q, \MM_q)$ tend to zero in $L^1$-norm as $q\to\infty$, which imply that  $(v_q, b_q)\to(v,b)$  solves \eqref{mhd} in weak sense. Moreover, $(v_q, b_q)$ satisfies
 $$(v_q, b_q)\equiv (v^{(1)}, b^{(1)})~{\rm on} ~[0, 2T+2\tau_q],\quad (v_q, b_q)\equiv (v^{(2)}, b^{(2)})~{\rm on} ~[3T-\tau_q, 1], $$
where $\tau_q\to 0$ as $q\to\infty$. Therefore, we can obtain that
 $$(v, b)\equiv (v^{(1)}, b^{(1)})~{\rm on} ~[0, 2T],\quad (v, b)\equiv (v^{(2)}, b^{(2)})~{\rm on} ~[3T, 1]. $$
To quantify the estimates on the solutions $(v_q, b_q, p_q,\RR_q, \MM_q)$, we introduce parameters and the inductive procedure firstly.

\subsection{Parameters and inductive procedure}
To begin with, we introduce all parameters of this paper. Fixed $0<\e<1$, $1\le p<2$, let $C_0=2^{12}$ and $\gamma, \sigma, b, \alpha, \beta$ be positive constants such that
\begin{equation}\label{gamma}
\gamma=\min\{\tfrac{4}{3p}-\tfrac{2}{3}, \tfrac{2\e}{3}, \tfrac{1}{24}\},\quad \sigma=\tfrac{\gamma}{16},\quad \quad  b=\tfrac{4C_0}{\sigma}, \quad \beta=\tfrac{1}{b^3}, \qquad \alpha=\tfrac{\beta}{16}.
\end{equation}
Moreover, we define
\begin{align}\label{10}
    \lambda_q \coloneq  \big\lceil a^{(b^q)}\big\rceil ,\quad\tau_q:=T\lambda^{-15}_q,  \quad \delta_q \coloneq \lambda_2^{3\beta}\lambda_q^{-2\beta}, \quad q\ge 1.
\end{align}
It is easy to verify that $\delta_1$ and  $\delta_2$  are  large numbers which could bound general initial data by setting $a$ sufficiently large, and for $q\geq3$, $\delta_q$ for is small and tend zero as $q\rightarrow\infty$.

To employ induction, we suppose that the solution $(v_q, b_q, p_q, \RR_q, \MM_q)$ of \eqref{e:subsol-euler} satisfies the following conditions: For all $q\ge 1$,
\begin{align}
     &\|(v_q,~b_q)\|_{L^2_{t,x}}+\|(v_q,~b_q)\|_{L^p_tL^\infty_x}+\|(v_q,~b_q)\|_{L^1_t C^{1-\e}_x}\le\sum_{i=1}^q\delta^{1/2}_i,\label{e:vq-LpLinftyC1-}\\
      &\|(v_q, b_q)\|_{L^\infty_tH^3_x} \le  \lambda^6_q ,
    \label{e:vqbq-H3}\\
    &\|(\RR_q, \MM_q)\|_{L^\infty_tH^3_x} \le  \lambda^{6}_q ,\,\,\|(\RR_q,~\MM_q)\|_{L^1_{t,x}} \le \delta_{q+1}\lambda^{-6\alpha}_q,
    \label{e:RR_q-L1}
\\
&(v_q, b_q)\equiv (v^{(1)}, b^{(1)}) \,\,\,\text {on}\,\,\,[0, 2T+2\tau_q] \,\,\,\text{and} \,\,\,(v_q, b_q)\equiv (v^{(2)}, b^{(2)}) \,\,\, \text {on}\,\,\, [3T-\tau_{q}, 1].\label{e:initial}
\end{align}
 The following proposition shows that there exists a solution  of \eqref{e:subsol-euler}
satisfying the above inductive conditions \eqref{e:vq-LpLinftyC1-}--\eqref{e:initial} with $q$ replaced by $q+1$, which guarantees the iteration proceeds successfully.

\begin{prop}
\label{p:main-prop}Fix $1\le p<2$, $0<\e<1$. Assume that
 $(v_q,b_q,p_q,\RR_q, \MM_q)$ solves
\eqref{e:subsol-euler} with initial data $(\vin,\bin)\in H^3(\mathbb T^3)$ and satisfies \eqref{e:vq-LpLinftyC1-}--\eqref{e:initial}.
Then there exists a solution $(v_{q+1},b_{q+1}, p_{q+1},\RR_{q+1}, \MM_{q+1})$ of \eqref{e:subsol-euler} satisfying
\eqref{e:vq-LpLinftyC1-}--\eqref{e:initial}
with $q$ replaced by $q+1$, and such that
\begin{align}
        \|(v_{q+1} - v_q,~b_{q+1} - b_q)\|_{L^2_{t,x}\cap( L^p_t L^\infty_x) \cap (L^1_t C^{1-\e}_x)} &\le\delta_{q+1}^{1/2}.
        \label{e:velocity-diff}
\end{align}
\end{prop}

Next, we show that Proposition \ref{p:main-prop} immediately implies Theorem \ref{t:main}.
\subsection{Proof of Theorem \ref{t:main} }
To start the iterative procedure, we need to construct $(v_1, b_1, p_1, \RR_1, \MM_1)$ firstly. Define a smooth cut-off function $\chi(t)$ such that $\text{supp}\chi=[-1, 3T-\tau_1]$ and $\chi|_{[0,2T+2\tau_1]}=1$, we construct $(v_1, b_1)$ by gluing $(v^{(1)}, b^{(1)})$ and $(v^{(2)}, b^{(2)})$ as follows
\begin{align*}
(v_1, b_1)=(\chi v^{(1)}+(1-\chi)v^{(2)}, \chi b^{(1)}+(1-\chi)b^{(2)}).
\end{align*}
One can easily verify that $(v_1, b_1)$ is divergence-free and  satisfies
\begin{equation}\label{v1b1}
\begin{aligned}
&\|(v_1,~b_1)\|_{L^2_{t,x}}+\|(v_1,~b_1)\|_{L^p_tL^\infty_x}+\|(v_1,~b_1)\|_{L^1_t C^{1-\e}_x}\le\delta^{1/2}_2\lambda^{-4\alpha}_1,\\
&\|(v_1,~b_1)\|_{L^\infty_t H^3_x}\le\lambda_1,
\end{aligned}
\end{equation}
for sufficient large $a$, which implies that $(v_1, b_1)$ satisfies \eqref{e:vq-LpLinftyC1-}--\eqref{e:vqbq-H3} for $q=1$. Noting the support of $\chi(t)$, we have, for $t\in [0, 2T+2\tau_1]\cup [3T-\tau_1, 1]$, $\RR_1=\MM_1=0$, and for $t\in (2T+2\tau_1, 3T-\tau_1)$,
\begin{align}
    \RR_1 \coloneq&
        \del_t \chi \mathcal R(v^{(1)}-v^{(2)} ) - \chi(1-\chi)(v^{(1)}-v^{(2)}  )\ootimes (v^{(1)}-v^{(2)}  )\notag\\
        &+\chi(1-\chi)(b^{(1)}-b^{(2)} )\ootimes(b^{(1)}-b^{(2)} ), \label{RR_1}\\
        \MM_1 \coloneq&
        \del_t \chi \mathcal R_a(b^{(1)}-b^{(2)}  ) - \chi(1-\chi)(v^{(1)}-v^{(2)}   )\otimes (b^{(1)}-b^{(2)} )\notag\\
        &+\chi(1-\chi)(b^{(1)}-b^{(2)} )\otimes (v^{(1)}-v^{(2)} ), \label{MM_1}\\
    p_1 \coloneq & \chi p^{(1)}+(1-\chi)p^{(2)} - \chi(1-\chi)\big( |v^{(1)}-v^{(2)}  |^2  - \int_{\mathbb T^3} |v^{(1)}-v^{(2)} |^2 \dd x\notag\\
    &-|b^{(1)}-b^{(2)} |^2  + \int_{\mathbb T^3} |b^{(1)}-b^{(2)} |^2 \dd x\big)\nonumber,
\end{align}
here and below, $v\ootimes u:=v\otimes u-(v\cdot u)\rm{Id}_{3\times3}$. $\mathcal R$ and $\mathcal{R}_{a}$ are  matrix-valued right inverse of the divergence operators introduced in Section~\ref{inversedivergence}.

Using estimates \eqref{e:vq-LpLinftyC1-} and \eqref{e:vqbq-H3} for $q=1$, we can obtain by the definition of $(\RR_1, \MM_1)$ in \eqref{RR_1}--\eqref{MM_1} that
\begin{align*}
\|(\RR_1, \MM_1)\|_{L^\infty_t H^3_x}\lesssim &\|(v^{(1)}, v^{(2)}, b^{(1)}, b^{(2)})\|_{L^\infty_t H^2_x}+\|(v^{(1)}, v^{(2)}, b^{(1)}, b^{(2)})\|_{L^\infty_t H^3_x}^2
<\lambda^{4}_1,\\
\|(\RR_1, \MM_1)\|_{L^1_{t,x}}\lesssim &\|(v^{(1)}, v^{(2)}, b^{(1)}, b^{(2)})\|_{L^2_{t,x}}+\|(v^{(1)}, v^{(2)}, b^{(1)}, b^{(2)})\|^2_{L^2_{t,x}}\\
\lesssim&\delta_2\lambda^{-8\alpha}_1\le\delta_2\lambda^{-6\alpha}_1,
\end{align*}
and
\[(v_1, b_1)\equiv (v^{(1)}, b^{(1)}) \,\,\,\text {on}\,\,\,[0, 2T+2\tau_1] \,\,\,\text{and} \,\,\,(v_1, b_1)\equiv (v^{(2)}, b^{(2)}) \,\,\, \text {on}\,\,\, [3T-\tau_1, 1].\]
Utilizing Proposition \ref{p:main-prop} inductively, we obtain a sequence of solutions $\{(v_q,b_q, p_q,\RR_q,\MM_q)\}$ of system \eqref{e:subsol-euler} with the inductive estimates \eqref{e:vq-LpLinftyC1-}--\eqref{e:initial}. By the definition of $\delta_q$, one can easily deduce that $\sum_{i=1}^{\infty}\delta_i$ converges to a finite number. This fact combined with \eqref{e:velocity-diff}
implies that $\{(v_q,b_q, p_q,\RR_q,\MM_q)\}$ is a Cauchy sequence in $L^2([0,1]; L^2(\TTT^3))\cap L^p([0,1]; L^\infty(\TTT^3))\cap L^1([0,1]; C^{1-\e}(\TTT^3))$. Thus, there exists a limit function $(v,b)$ satisfying
$$(v,b)\in L^2([0,1]; L^2(\TTT^3))\cap L^p([0,1]; L^\infty(\TTT^3))\cap L^1([0,1]; C^{1-\e}(\TTT^3))$$
and solves \eqref{mhd} in weak sense because $\|(\RR_q, \MM_q)\|_{L^1_{t,x}}\rightarrow 0$ as $q\to\infty$. Note the fact that $\tau_q\to 0$ when $q\to\infty$, one can deduce by \eqref{e:initial} that
\begin{align*}
 (v, b)\equiv (v^{(1)}, b^{(1)}) \,\,\,\text {on}\,\,\,[0, 2T] \,\,\,\text{and} \,\,\,(v, b)\equiv (v^{(2)}, b^{(2)}) \,\,\, \text {on}\,\,\, [3T, 1].
 \end{align*}

 Next, we show that Theorem \ref{t:main} implies  Theorem \ref{t:main0}.
\subsection{Proof of Theorem \ref{t:main0}}

For a given weak solution $(v, b)\in L^p([0, T]; L^\infty(\TTT^3))$ with initial data $(v_0, b_0)$,
there exists $t_0\in (0, T)$ such that $(v, b)|_{t=t_0}\in L^\infty(\TTT^3)$. By the well-posedness results on MHD system, we have a unique local strong solution $(v^{1}, b^{1})$ on $[t_0,t_{\rm local}]\subset[0, T]$ with initial data $(v(t_0), b(t_0))$ and fixed $0<4\varepsilon<t_{\rm local}-t_0$,
$$(v^{1}(t_0+\varepsilon), b^{1}(t_0+\varepsilon))\in H^3(\TTT^3).$$
Taking advantage of Theorem \ref{t:main} on $[t_0+\varepsilon, t_{\rm local}]$ instead of $[0, \widetilde{T}]$, choosing $(v^{(1)}, b^{(1)})=(v^{1}, b^{1})$ and $(v^{(2)}, b^{(2)})$ to be the shear flows  in $[t_0+\varepsilon,t_{\rm local}]$ such that
 \begin{align*}
v^{(2)}=(0,0,Ae^{-4\pi^2 t}\sin(2\pi x_1)),~~
  b^{(2)}=(0,Ae^{-4\pi^2 t}\sin(2\pi x_1),0),
  \end{align*}
for any positive constant $A$, one can obtain a weak solution $(\widetilde{v}, \widetilde{b})\in L^p([0, T]; L^\infty(\TTT^3)\cap L^2([0, T]\times \TTT^3)$ with initial data $(v_0, b_0)$ such that
$$(\widetilde{v}, \widetilde{b})\equiv(v, b)~{\rm on}~[0, t_0], ~(\widetilde{v}, \widetilde{b})\equiv (v^1, b^1)~{\rm on}~[t_0, t_0+2\varepsilon],~(\widetilde{v}, \widetilde{b})\equiv (v^{(2)}, b^{(2)})~{\rm on}~[t_0+3\varepsilon, T].$$
The above relationship shows that for large enough $A$,
 \[\|(\widetilde{v}(t_{\rm local}), \widetilde{b}(t_{\rm local}))\|^2_{L^2_x}=\|({v^{(2)}}(t_{\rm local}),{b^{(2)}}(t_{\rm local}))\|^2_{L^2_x}>2\|(v_0, b_0)\|^2_{L^2_x},\]
 This shows that $(\widetilde{v}, \widetilde{b})$ is a non-Leray-Hopf weak solution and $(\widetilde{v}, \widetilde{b})\neq (v, b)$ by adjusting $A$. Therefore, we finish the proof of Theorem \ref{t:main0}.

{Next, we are devoted to  proving Proposition \ref{p:main-prop} in the rest of this paper.} The proof of Proposition \ref{p:main-prop} consists of the following two steps: Gluing: $(v_q, b_q)\mapsto (\bar{v}_q, \bar{b}_q)$ in Section \ref{Gluing procedure} and convex integration scheme: $(\bar{v}_q, \bar{b}_q) \mapsto (v_{q+1}, b_{q+1})$ in Section \ref{perturbation}.

\section{Gluing procedure}\label{Gluing procedure}
\subsection{Classical exact  flows}\label{stability}
\label{ss:exact}
Before giving exact solutions $(\vex_i, \bex_i)$, we introduce the local well-posedness result on the MHD equations for smooth initial data. For the detailed proof, readers can refer to \cite{2018Wild}, in which the well-posedness of the Navier-Stokes equations have been established.
\begin{prop}\label{well-posed}Assume that the mean-free initial data $(\vin, \bin)\in H^3(\TTT^3)$, then there exists a universal positive constant $c\le 1$ such that system \eqref{mhd} possesses a unique strong solution $(v, b)$ on $[0, T_0]$ satisfying the following estimates:
\begin{align}
&\sup_{t\in [0, T_0]}(\|v(t)\|^2_{L^2_x}+\|b(t)\|^2_{L^2_x})+2\int_0^{T_0}(\|v\|^2_{\dot H^1_x}+\|b\|^2_{\dot H^1_x})\dd x\le \|\vin\|^2_{L^2_x}+\|\bin\|^2_{L^2_x},\\
&\sup_{t\in [0, T_0]}(\|v(t)\|^2_{H^3_x}+\|b(t)\|^2_{H^3_x})\le 2(\|\vin\|_{H^3_x}+\|\bin\|_{H^3_x}),
\end{align}
and for any $N\ge 0$, $M=0, 1$,
\begin{align}
\sup_{t\in (0, T_0]}|t|^{\frac{N}{2}+M}\|\partial^M_t D^N v(t)\|_{H^3_x}+\sup_{t\in (0, T_0]}|t|^{\frac{N}{2}+M}\|\partial^M_t D^N b(t)\|_{H^3_x}\lesssim \|\vin\|_{H^3_x}+ \|\bin\|_{H^3_x},
\end{align}
where the implicit constant depends on $N$ and $M$, and
\[T_0=\frac{c}{(\|\vin\|_{H^3_x}+\|\bin\|_{H^3_x})(1+\|\vin\|_{L^2_x}+\|\bin\|_{L^2_x})}.\]
\end{prop}
In view of Proposition \ref{well-posed}, there exists a unique strong solution $(\vex_i, \bex_i)(i\ge 1)$ of  the following system on $[t_i, t_{i+2}]$:
\begin{equation}\label{eulervi}
\left\{ \begin{alignedat}{-1}
&\del_t \vex_i -\Delta \vex_i+\div(\vex_i\otimes  \vex_i)  +\nabla \pex_i   =  \div (\bex_{i}\otimes \bex_{i}) ,
\\
&\del_t \bex_{i}-\Delta \bex_i +\div (\vex_{i}\otimes \bex_{i})   = \div (\bex_{i}\otimes \vex_{i})  ,
\\
&  \nabla \cdot \vex_{i} = 0,~~\nabla \cdot \bex_{i} = 0,
  \\
& \vex_i|_{t=t_i}= v_q(\cdot,t_i),~~\bex_i|_{t=t_i} = b_q(\cdot,t_i).
\end{alignedat}\right.
\end{equation}
Furthermore, we can obtain the following estimates of $(\vex_i, \bex_i)$ on $[t_i, t_{i+2}]$:
\begin{cor}\label{vibiHs}Let $(\vex_i, \bex_i)$ be the solution to \eqref{eulervi} on $[t_i, t_{i+2}]$, then it holds that
\begin{align}
&\|\vex_i(t),\bex_i(t)\|_{H^{3}_x}\lesssim \|v_q(t_i), b_q(t_i)\|_{H^{3}_x}\lesssim  \lambda^{6}_{q}, \quad \forall t_i\le t\le t_{i+2},\label{vibiH3},
\end{align}
and for $\forall t\in [t_i+\frac{\tau_q}{3}, t_{i+2}]$, $N\ge 0$ and $M=0, 1$,
\begin{align}\label{vibitimespace}
&\|\partial^M_t D^N \vex_i(t)\|_{H^3_x}+\|\partial^M_t D^N \bex_i(t)\|_{H^3_x}\lesssim \tau^{-\frac{N}{2}-M}_q\|v_q(t_i), b_q(t_i)\|_{H^{3}_x}\lesssim \lambda^{\tfrac{15N}{2}+7}_q,
\end{align}
where implicit constant depends on $N$.
\end{cor}
Next, we show the stability between the approximate solution $(v_q, b_q)$ and the exact solution $(\vex_i, \bex_i)$ on $[t_i, t_{i+2}]$.
\begin{prop}\label{vi-vq}
Let $i\ge 1$ and $t\in [t_{i}, t_{i+2}]$. We have
\begin{align}
&\|(\vex_i - v_q,\bex_i - b_q)\|^2_{L^2}\lesssim\lambda^{-3}_q,\label{gai-vi-vqL2}\\
&\|(\vex_i - v_q,\bex_i - b_q)\|_{L^{\infty}([t_i, t_{i+2}],B^{0}_{\infty,1})\cap L^{1}([t_i, t_{i+2}],B^2_{\infty,1})}\lesssim\lambda^{-9}_q,\label{vi-vqL2}\\
&\|(\vex_i-v_q)(t)\|_{B^{-1}_{1,1}}+\|(\bex_i -b_q)(t)\|_{B^{-1}_{1,1}} \lesssim\int_{t_{i}}^{ t_{i+2}}\|\RR_q,\MM_q\|_{ B^{0}_{1,1}}\dd s.\label{e:deriv-vector-potential}
 \end{align}

 \end{prop}
 \begin{proof}To begin with, we recall that the equations for the difference are
\begin{equation}\label{e:eqn-for-vex_i-v_ell}
\left\{ \begin{alignedat}{-1}
&\del_t (\vex_i - v_q)-\Delta (\vex_i - v_q)+v_q\cdot\nabla (\vex_i - v_{q}) + (\vex_i-v_q)\cdot\nabla \vex_i + \nabla (\pex_i - p_{q}) \\
&\qquad=b_{q} \cdot\nabla (\bex_i - b_{q}) + (\bex_i-b_q)\cdot\nabla \bex_i-\div \RR_{q},\\
&\del_t (\bex_i - b_{q})-\Delta (\bex_i - b_{q}) +v_{q} \cdot\nabla (\bex_i - b_{q}) + (\vex_i-v_q)\cdot\nabla \bex_i \\
   &\qquad=b_{q} \cdot\nabla (\vex_i - v_{q}) + (\bex_i-b_q)\cdot\nabla \vex_i - \div \MM_{q},\\
   &(\vex_i-v_q)(\cdot,t_{i})=0,~~
   (\bex_i-b_q)(\cdot,t_{i})=0.
\end{alignedat}\right.
\end{equation}
Employing energy method and $\div \vex_i=\div \bex_i=\div v_q=\div b_q=0$, we deduce by  the Gr\"{o}nwall inequality, estimates \eqref{e:vqbq-H3}  and  \eqref{vibiH3} that, for $t\in [t_{i}, t_{i+2}]$,
\begin{align*}
\|(\vex_i - v_q)(t)\|^2_{L^2_x}+\|(\bex_i - b_{q})(t)\|^2_{L^2_x}\le &Ce^{C\int_{t_{i}}^t(\|\nabla \vex_i\|^2_{L^\infty_x}+\|\nabla \bex_i\|^2_{L^\infty_x})\dd s}(\int_{t_{i}}^t\|\RR_q\|^2_{L^2_x}+\|\MM_q\|^2_{L^2_x}\dd s)\\
\le& Ce^{C\tau_q\lambda^{12}_q}\lambda^{12}_{q}\tau_q\le Ce^{C\lambda^{-3}_q}\lambda^{-3}_q\le C\lambda^{-3}_q.
\end{align*}
Furthermore, using the estimates in Besov spaces for the heat equation in \cite{WDY}, we have
\begin{align*}
&\|(\vex_i - v_q,\bex_i - b_q)\|_{L^{\infty}([t_i, t_{i+2}],B^{0}_{\infty,1})\cap L^{1}([t_i, t_{i+2}],B^2_{\infty,1})} \\
\lesssim& \int_{t_{i+2}}^{t_i}\|(\vex_i,v_q, \bex_i,b_{q})\|_{H^3_x}\|(\vex_i - v_q, \bex_i - b_{q})\|_{B^{1}_{\infty,1}}\dd t+\int_{t_{i+2}}^{t_i}\|(\RR_q, \MM_q)\|_{B^1_{\infty,1}}\dd t\\
\lesssim& \tau^{\frac{1}{2}}_q\|(\vex_i,v_q, \bex_i,b_{q})\|_{L^\infty([t_i, t_{i+2}],H^3_x)}\|(\vex_i - v_q, \bex_i - b_{q})\|_{L^2([t_i, t_{i+2}], B^{1}_{\infty,1})}\dd t+\int_{t_{i+2}}^{t_i}\|(\RR_q, \MM_q)\|_{B^{1}_{\infty,1}}\dd t.
\end{align*}
Due to $\tau^{\frac{1}{2}}_q\|(\vex_i,v_q, \bex_i,b_{q})\|_{L^\infty([t_i, t_{i+2}],H^3_x)}\lesssim \lambda^{-\frac{3}{2}}_q\ll1$, we obtain from the above inequality that
\begin{align}\label{B0infty1}
\|(\vex_i - v_q,\bex_i - b_q)\|_{L^{\infty}([t_i, t_{i+2}],B^{0}_{\infty,1})\cap L^{1}([t_i, t_{i+2}],B^2_{\infty,1})}
\lesssim \int_{t_i}^{t_{i+2}}\|(\RR_q, \MM_q)\|_{ B^{1}_{\infty,1}}\dd s\lesssim\lambda^{-9}_q.
\end{align}
In a similar way as deriving the above inequality, we obtain \eqref{e:deriv-vector-potential},  which is used to estimate the Reynolds and magnetic stresses in gluing stage.
\end{proof}

Taking advantage of the exact solutions $\{(\vex_i, \bex_i)\}_{i\ge 0}$, we construct a new approximate solution $(\vv_q, \bb_q)$ by gluing these exact solutions. Noting the stability estimates between  $(\vex_i, \bex_i)$ and $(v_q, b_q)$, one would expect that the glued solution $(\vv_q, \bb_q)$ satisfies estimates \eqref{e:vq-LpLinftyC1-}-\eqref{e:vqbq-H3} up to a universal constant.
\subsection{Gluing exact  flows}\label{gluing}
We define 
$$t_0=2T, t_i=t_0+i\tau_q,  N_q:=\lambda^{15}_q-1,$$ 
and the time intervals $I_i,J_i$ ($1\le i\le N_q+1)$ by
\begin{align}
    I_i \coloneq \Big[ t_i + \frac{\tau_q}3,\ t_i+\frac{2\tau _q}3\Big] ,\quad
     J_i \coloneq \Big(t_i - \frac{\tau_q}3,\ t_i+\frac{\tau _q}3\Big).
\end{align}
Because $\tau_q=T\lambda^{-15}_q$, it is easy to verify that $3T-\tau_q=t_{N_q}$. Let  $\{ \chi_i\}_{i=1}^{N_{q}+1}$  be a partition of unity such that
\begin{align}
       &\text{supp} \chi_1=[-1, 2T+\tfrac{5\tau_q}{3}], \qquad\chi_i |_{\big[0, 2T+\tfrac{4\tau_q}{3}\big]} =1, \qquad\|  \del_t^N \chi_1\|_{C^0_t} \lesssim \tau_q^{-N},\ \    \label{e:chi_i-properties}\\
           & \text{supp} \chi_i=I_{i-1} \cup J_i \cup I_{i}, \quad\quad\,\,\,\, \chi_i |_{J_i} =1, \qquad\qquad\quad\,\,  \|  \del_t^N \chi_i\|_{C^0_t} \lesssim \tau_q^{-N},\ \   2 \le i\le N_q,\label{e:chi_i-properties}\\
       &\text{supp} \chi_{N_q+1}= [3T-\tfrac{2\tau_q}{3},2], \quad\,\, \chi_i |_{\big[3T-\tfrac{\tau_q}{3}, 1\big]} =1, \qquad\,\,\|  \del_t^N \chi_{N_q+1}\|_{C^0_t} \lesssim \tau_q^{-N}.  \label{gai-e:chi_i-properties}
\end{align}
In particular, for $|i-j|\ge2$, $\text{supp} \chi_i \cap \text{supp} \chi_{j} =\emptyset$. For $1\le i\le N_q-1$, $(\vex_{i},\bex_{i}, \pex_i)$ is the solution of system \eqref{eulervi} on $[t_i, t_{i+2}]$. And we define
$$(\vex_{0},\bex_{0}, \pex_{0}):=(v_q, b_q, p_q)~{\rm on}~[0, 2T+2\tau_q],\,\,(\vex_{N_q},\bex_{N_q}, \pex_{N_q}):=(v_q, b_q, p_q)~{\rm on}~[3T-\tau_q, 1].$$
The glued velocity, magnetic fields and pressure $(\vv_q, \bb_q, \ppp_q)$ are constructed as follows:
\begin{align}
    \vv_q(x,t) &\coloneq \sum_{i=0}^{N_q} \chi_{i+1}(t) \vex_{i}(x,t),  \,\forall \,t\in [0,1], \label{e:vv_q}\\
    \bb_q(x,t) &\coloneq \sum_{i=0}^{N_q} \chi_{i+1}(t) \bex_{i}(x,t),  \,\forall\, t\in [0,1], \label{e:bb_q}\\
    \ppp_q(x,t) &\coloneq \sum_{i=0}^{N_q}\chi_{i+1}(t) \pex_i(x,t),  \,\forall \,t\in [0,1]. \label{e:pp_q}
\end{align}
Since $\chi_i(t)$ only depends on time variable,  $\vv_q$ and $\bb_q$ are still divergence-free, and
\begin{equation}\nonumber
\begin{aligned}
 t\in I_i,~1\le i\le N_q&\implies
\left\{ \begin{alignedat}{-1}
\vv_q(x,t)&=\vex_{i-1}(x,t)\chi_{i}(t) + \vex_{i}(x,t)\chi_{i+1}(t),\\
\bb_q(x,t)&=\bex_{i-1}(x,t)\chi_{i}(t)+\bex_{i}(x,t)\chi_{i+1}(t),
\end{alignedat}\right.  \\
t\in J_i,~1\le i\le N_q&\implies \vv_q (x,t)=\vex_{i-1}(x,t),~\bb_q (x,t)=\bex_{i-1}(x,t),\\
t\in [0, 2T+\tfrac{4\tau_q}{3}]\quad&\implies \vv_q (x,t)=\vex_{0}(x,t),~\bb_q (x,t)=\bex_{0}(x,t),\\
t\in [3T-\tfrac{\tau_q}{3}, 1]\quad&\implies \vv_q (x,t)=\vex_{q}(x,t),~\bb_q (x,t)=\bex_{q}(x,t).
\end{aligned}
\end{equation}
Therefore, for  $ t\in \mathop{\cup}\limits_{i=1}^{N_q}J_i \cup [0, 2T+\frac{4\tau_q}{3}]\cup [3T-\tfrac{\tau_q}{3}, 1]$, $(\vv_q, \bb_q)$ is an exact MHD flow. For $t\in I_i(i\ge 1)$, $(\vv_q,\bb_q)$ solves the equations
\begin{equation}\label{e:eqn-for-vex_i-v_ell}
\left\{ \begin{alignedat}{-1}
&\del_t \vv_q+\div (\vv_q\otimes \vv_q)  +\nabla \pp_q   =  \div (\bb_q\otimes \bb_q)+ \div\RRR_q,\\
&\del_t \bb_q+\div (\vv_q\otimes \bb_q) =\div (\vv_q\otimes \vv_q)+ \div\MMM_q,\\
&\nabla\cdot\vv_q=\nabla\cdot\div\bb_q=0,
\end{alignedat}\right.
\end{equation}
where
\begin{align}
    \RRR_q \coloneq&
        \del_t \chi_i \mathcal R(\vex_{i-1}-\vex_{i} ) - \chi_i(1-\chi_i)(\vex_{i-1}-\vex_{i} )\ootimes (\vex_{i-1}-\vex_{i} )\notag\\
        &+\chi_i(1-\chi_i)((\bex_{i-1}-\bex_{i} )\ootimes (\bex_{i-1}-\bex_{i} )), \label{RRR_q}\\
        \MMM_q \coloneq&
        \del_t \chi_i\mathcal{R}_{a}(\bex_{i-1}-\bex_{i} ) - \chi_i(1-\chi_i)(\vex_{i-1}-\vex_{i})\otimes (\bex_{i-1}-\bex_{i} )\notag\\
        &+\chi_i(1-\chi_i)((\bex_{i-1}-\bex_{i})\otimes (\vex_{i-1}-\vex_{i} )), \label{MMM_q}\\
    \pp_q   \coloneq &\ppp_q - \chi_i(1-\chi_i)\big( |\vex_{i-1}-\vex_{i}|^2  - \int_{\mathbb T^3} |\vex_{i-1}-\vex_{i}|^2 \dd x\notag\\
    &-|\bex_{i-1}-\bex_{i}|^2  + \int_{\mathbb T^3} |\bex_{i-1}-\bex_{i}|^2 \dd x\big).
\end{align}
Because ${\chi}_{N_q+1}$ only depends on $t$ and $ B^0_{\infty, 1}\hookrightarrow L^\infty_{x}, B^{1-\e}_{\infty, 1}\hookrightarrow  C^{1-\e}$, one can easily obtain the following proposition  by Proposition \ref{vi-vq}:
\begin{prop}[Estimates for $(\vv_q, \bb_q)$]\label{estimate-vvq}
Let $(\vv_q, \bb_q)$ be defined by \eqref{e:vv_q}. There exist the following estimates:
\begin{align}
 &\|(\vv_{q},\bb_{q} )\|_{L^2_{t,x}} +\|(\vv_{q},\bb_{q} )\|_{L^p_tL^\infty_x}+\|(\vv_{q},\bb_{q}) \|_{L^1_t C^{1-\e}_x} \le  C\lambda^{-\frac{3}{2}}_q+\sum_{i=1}^q\delta^{1/2}_i,\label{e:vv_q-L2}\\
 &\|(\vv_{q},\bb_{q} )\|_{L^\infty_t H^{3}_x} \lesssim \lambda^6_q.\label{e:vv_q-Hs}
\end{align}
Moreover, we have
\begin{align}
(\vv_{q},\bb_{q})\equiv (v^{(1)}, b^{(1)})\,{\rm on}\,[0, 2T+\tfrac{4\tau_q}{3}]\times\TTT^3, \, \,\,(\vv_{q},\bb_{q})\equiv (v^{(2)}, b^{(2)})\,{\rm on}\,[3T-\tfrac{\tau_{q}}{3}, 1]\times\TTT^3\label{e:vv_q-initial}.
 \end{align}
\end{prop}

Next, we show that $(\RRR_q, \MMM_q)$ is a small perturbation of $(\RR_q, \MM_q)$ in $L^1_{t,x}$ space, and we give time-space estimates of $(\RRR_q, \MMM_q)$, which will be used to bound perturbation $(w_{q+1}, d_{q+1})$ in convex integration scheme.
\begin{prop}[Estimates for $(\RRR_q, \MMM_q)$]\label{MMM_qRRR_qL1}
    \begin{align}
    & \|(\RRR_q,\MMM_q)\|_{L^1_{t,x}}\le \delta_{q+1}{\lambda^{-4\alpha}_q}, \\
        &\|(\RRR_q,\MMM_q)\|_{L^\infty_t H^3_x}
       \le \lambda^{22}_q,\\
        &\|(\partial^M_t D^N \RRR_q, \,\partial^M_t D^N\MMM_q)\|_{L^\infty_t H^3_x} \le\lambda^{55+15N}_q , \quad\forall N\ge0, \, M\in\{0,1\}.
    \end{align}
\end{prop}
\begin{proof}By the definition of $(\RRR_q, \MMM_q)$ in \eqref{RRR_q} and \eqref{MMM_q}, we have
\begin{align*}
\|(\RRR_q, \MMM_q)\|_{L^1_{t,x}}
\le&\sum_{i=1}^{N_q}\|\del_t \chi_i\|_{L^1_t}\|\mathcal R(\vex_{i-1}-\vex_{i} )\|_{L^\infty(I_i; L^1)}+\sum_{i=1}^{N_q}\|\del_t \chi_i\|_{L^1_t}\|\mathcal R_a(\bex_{i-1}-\bex_{i} )\|_{L^\infty(I_i; L^1)}\\
&+\|\vex_{i-1}-\vex_{i} \|^2_{L^\infty (I_i;L^2)}+\|\bex_{i-1}-\bex_{i} \|^2_{L^\infty (I_i;L^2)}).
\end{align*}
With aid of Proposition \ref{vi-vq}, choosing $0<\varepsilon <\frac{\alpha}{3}$ and  sufficient large $a$, we deduce from the above inequality and \eqref{e:RR_q-L1} that,
\begin{align*}
\|(\RRR_q,\MMM_q)\|_{L^1_{t,x}}
\le&C\sum_{i=1}^{N_q}\int_{t_{i}}^{ t_{i+2}}\|(\RR_q,\MM_q)\|_{ B^{0}_{1,1}}d s+C\lambda^{-3}_q\\
\le&C\|(\RR_q,\MM_q)\|^{1-{\varepsilon}}_{L^1_{t,x}}\|(\RR_q,\MM_q)\|^{{\varepsilon}}_{L^1_t H^3_x}+C\|(\RR_q,\MM_q)\|_{L^1_{t,x}}+C\lambda^{-3}_q\\
\le&C(\delta_{q+1}\lambda^{-6\alpha})^{1-{\varepsilon}}(\lambda^5_q)^{{\varepsilon}}+C\lambda^{-3}_q\le \delta_{q+1}\lambda^{-4\alpha}_q.
\end{align*}
In view of Corollary \ref{vibiHs}, one obtains that
\begin{align*}
\|(\RRR_q,\MMM_q)\|_{L^\infty_t H^3_x}\lesssim &\tau^{-1}_q\|(\vex_{i-1}-\vex_i, \bex_{i-1}-\bex_i)\|_{L^\infty(I_i; H^2)}
+\|(\vex_{i-1}-\vex_i, \bex_{i-1}-\bex_i)\|^2_{L^\infty(I_i; H^2)}
\le\lambda^{22}_q.
\end{align*}
Furthermore, noting the fact that $\RRR_q$ and $\MMM_q$ support on $\mathop{\cup}\limits_{i=1}^{N_q} I_i$, where $\{(\vex_i, \bex_i)\}_{i\ge 0}$ possess $C^\infty$ regularity, we obtain by \eqref{vibitimespace} that for $M=0,1$, $\forall N\ge 0$,
\begin{align*}
&\|(\partial^M_t D^N \RRR_q, \,\,\partial^M_t D^N \MMM_q)\|_{L^\infty_t H^3_x}\\
\lesssim&\tau^{-3}_q\sup_{1\le i\le N_q}\|\partial^M_t(\vex_{i-1}-\vex_{i})\|_{L^\infty(I_i; H^{N+2})}+\tau^{-3}_q\sup_{1\le i\le N_q}\|\partial^M_t(\vex_{i-1}-\vex_{i})\|_{L^\infty(I_i; H^{N+2})}\\
&+\tau^{-2}_q(\|\partial^M_tD^N(\vex_{i-1}, \vex_{i}, \bex_{i-1}, \bex_{i})\|^2_{L^\infty (I_i; H^3)}+\|(\vex_{i-1}, \vex_{i},\bex_{i-1}, \bex_{i})\|^2_{L^\infty (I_i; H^3)})\\
&+\tau^{-2}_q(\|\partial^M_t(\vex_{i-1}, \vex_{i}, \bex_{i-1}, \bex_{i})\|^2_{L^\infty (I_i; H^3)}+\|D^N(\vex_{i-1}, \vex_{i}, \bex_{i-1}, \bex_{i})\|^2_{L^\infty (I_i; H^3)})\\
\le& \lambda^{55+15N}_q,
\end{align*}
where the last inequality holds by sufficient large $a$.
\end{proof}

\section{Convex integration step: construction and estimates for the perturbation}\label{perturbation}
To reduce the size of the Reynolds and magnetic stresses at $q+1$ level, the construction of the perturbation
heavily depends on the two geometric lemmas. For the convenience of readers, we review the two geometric lemmas firstly.
\subsection{Geometric lemmas}
\begin{lem}[First Geometric Lemma\cite{2Beekie}]\label{first L}Let $B_{\delta}(0)$ denote the ball of radius $\delta$ centered at $0$ in the space of $3\times3$ skew-symmetric matrices.
There exists a set $\Lambda_b\in \mathbb{S}^2\cap\mathbb{Q}^3$ that consists of vectors $k$ with associated orthonormal basis $(k,\bar{k},\bar{\bar{k}}),~\epsilon_b>0$ and smooth positive functions $a_{b,k}:B_{\epsilon_b}(0)\rightarrow\mathbb{R},$ such that, for $M_b\in B_{\epsilon_b}(0)$ we have the following identity:
$$M_b=\sum_{k\in\Lambda_b}a_{b,k}(M_b)(\bar{k}\otimes\bar{\bar{k}}-
\bar{\bar{k}}\otimes\bar{k}).$$
\end{lem}

\begin{lem}[Second Geometric Lemma\cite{2Beekie}]\label{first S}Let $B_{\sigma}(0)$ denote the ball of radius $\sigma$ centered at $\rm Id$ in the space of $3\times3$ symmetric matrices.
There exists a set $\Lambda_u\subset\mathbb{S}^2\cap\mathbb{Q}^3$ that consists of vectors $k$ with associated orthonormal basis $(k,\bar{k},\bar{\bar{k}}),~\epsilon_u>0$ and smooth positive functions $a_{u,k}:B_{\epsilon_u}(\rm Id)\rightarrow\mathbb{R}$ such that, for $R_u\in B_{\epsilon_u}(\rm Id)$ we have the following identity:
$$R_u=\sum_{k\in\Lambda_u}a_{u,k}(R_u)\bar{k}\otimes\bar{k}.$$
\end{lem}
\begin{rem}
It is worth noting that, once we choose a vector $k$ in $\Lambda_u\cup\Lambda_b$, a unique triple $(k,\bar{k},\bar{\bar{k}})$ could be obtained, see \cite{2Beekie} for more details.
\end{rem}
\subsection{Amplitudes}
Let $\eta_q\in C_c^\infty([2T, 3T];[0,1])$  satisfy
\[  \text{supp}\, \eta_q =\big(2T+\tfrac{7\tau_q}{6}, 3T-\tfrac{\tau_q}{6}\big)\quad\text{and}\quad \eta_q(t)\equiv 1 \,\text{for}\,\,t\in \big[2T+\tfrac{4\tau_q}{3}, 3T-\tfrac{\tau_q}{3}\big].\]
It is easy to check that for $N\ge 1$, $\|\del_t ^N \eta_q\|_{C^0_t} \lesssim \tau_q^{-N}. $

Note that the geometric lemmas are valid for skew-symmetric matrices perturbed near $0$ matrix and symmetric matrices perturbed near $\rm Id$ matrix, we need the following smooth function $\chi$ introduced in  \cite{2Beekie} such that the stresses are pointwise small. More precisely: Let $\chi:[0,\infty)\to \mathbb{R}^{+}$ be a smooth function satisfying
\begin{equation}
\chi(z)=\left\{ \begin{alignedat}{-1}
&1, \quad 0\le z\le 1,\\
&z, \quad z\ge 2,
\end{alignedat}\right.
\end{equation}
with $z\le 2\chi(z)\le 4z$ for $z\in (1,2)$. Next, we define
\[\rho_b=\delta_{q+1}\lambda^{-2\alpha}_q\chi\Big(\Big\langle\tfrac{\MMM_q}{\delta_{q+1}\lambda^{-4\alpha}_q}\Big\rangle\Big),\,\,
{\rho_v=\delta_{q+1}\lambda^{-2\alpha}_q}\chi\Big(\Big\langle\tfrac{R_v}{\delta_{q+1}\lambda^{-3\alpha}_q}\Big\rangle\Big),\]
where
\[R_v=\RRR_q+\sum_{k\in \Lambda_b}\eta^2_q\rho_ba^2_{b,k}\Big(-\tfrac{\MMM_q}{\rho_b}\Big)(\bar{k}\otimes \bar{k}-\bar{\bar{k}}\otimes \bar{\bar{k}}).\]
Now we define amplitudes $a_{(v,k)}$ and $a_{(b,k)}$ as follows:
\begin{equation}\label{a(v,k)a(b,k)}
a_{(v,k)}:=\eta_{q}\rho^{\frac{1}{2}}_{v}a_{v,k}({\rm Id}-
\tfrac{R_v}{\rho_v}),\quad a_{(b,k)}:=\eta_{q}\rho^{\frac{1}{2}}_{b}a_{b,k}(-
\tfrac{\MMM_q}{\rho_b}),
\end{equation}
where $a_{b,k}$ and  $a_{v,k}$  stem from Lemma \ref{first L} and Lemma \ref{first S}, respectively.
\subsection{Constructing the perturbation}
Assume that $\phi:\mathbb{R}\rightarrow\mathbb{R}$ is a smooth cutoff function supported on the interval $[0,\frac{1}{2}]$. We normalize it in such a way that
$$\int_{\mathbb{R}}\phi^2\dd x=1.$$
For any small positive parameter $r$, we denote by $\phi_r$ the following rescaled functions:
$$\phi_r :=r^{-\frac{1}{2}}\phi (r^{-1}x).$$
We periodize $\phi_r$ so that  the resulting functions are periodic functions defined on $\mathbb{R}/ \mathbb{Z}=\mathbb{T}$.
Next, letting $k\in\Lambda_v\cup \Lambda_b$, we choose a large number  $N_{\Lambda}$ such that
$N_{\Lambda}{k},N_{\Lambda} \bar{k},N_{\Lambda}\bar{\bar{k}}\in \mathbb{Z}^3.$
Here and below, we use the following denotations:
$$\phi_{(\gamma,\tfrac{1}{2},k)}(x): =\phi_{\lambda^{-\gamma}_{q+1}}(\lambda^{\tfrac{1}{2}}_{q+1}N_{\Lambda}k\cdot x),~~g_{{(2, \iota)}}(t):=\phi_{\lambda^{-2}_{q+1}}(\lambda^{\sigma}_{q+1}t).$$
Without ambiguity, we denote $\phi_{(\gamma,\tfrac{1}{2},k)}(x)$ by $\phi_{(\gamma,\tfrac{1}{2},k)}$ only depending on spatial variable and denote $g_{{(2, \sigma)}}(t)$ by $g_{{(2, \sigma)}}$ only depending on time variable in the rest of this paper. Let $\psi=\frac{\dd^2}{\dd x^2}\Psi\in C^{\infty}(\mathbb{T})$. We define
$$\psi_k:=\psi(\lambda_{q+1}N_{\Lambda}k\cdot x)$$
 and normalize it such that $\|\psi_k\|_{L^2}=1$.

 The construction of perturbation are involved in the so-called \emph{intermittent shear flows} $\phi_{(\gamma,\tfrac{1}{2},k)}\bar{k}$ or $\phi_{(\gamma,\tfrac{1}{2},k)}\bar{\bar{k}}$. One can easily check that the intermittent shear flows have zero mean and support on sheets with size $\thickapprox\lambda^{-\gamma}_{q+1}\times 1\times 1$. In summary, we have the following estimates:
\begin{prop}\label{guji1}
For $p\in[1,\infty]$ and $m\in\mathbb{N}$, we have the following estimations
$$\big\|{D}^m_x\phi_{(\gamma,\tfrac{1}{2},k)}\big\|_{L^p_x(\TTT^3)}\leq \lambda^{\frac{m}{2}-\gamma(\frac{1}{p}-\frac{1}{2})}_{q+1} , \qquad\big\|\tfrac{\dd ^m}{\dd t}g_{{(2, \sigma)}}\big\|_{L^p_t([0,1])}\leq \lambda^{{\sigma m}-2(\frac{1}{p}-\frac{1}{2})}_{q+1},
$$
and the following estimates for the size of support of $\phi_{(\gamma,\tfrac{1}{2},k)}$:
\[\Big|{\rm supp} \big(\phi_{(\gamma,\tfrac{1}{2},k)}\big)\Big|\lesssim \lambda^{-\gamma}_{q+1}, \quad \Big|{\rm supp} \big(\phi_{(\gamma,\tfrac{1}{2},k)}\phi_{(\gamma,\tfrac{1}{2},k')}\big)\Big|\lesssim \lambda^{-2\gamma}_{q+1}\,\, \,{\rm if}\,\,k\neq k'.\]
\end{prop}

Based on the above intermittent shear flows, we define the \emph{principal parts} of the perturbations $w^{(p)}_{q+1}$ and $d^{(p)}_{q+1}$ as follows:
\begin{align*}
&w^{(p)}_{q+1}
:=\sum_{k\in \Lambda_v}a_{(v,k)}\phi_{{(\gamma,\tfrac{1}{2},k)}}g_{{(2,\sigma)}}\psi_k\bar{k}+\sum_{k\in \Lambda_b}a_{(b,k)}\phi_{{(\gamma,\tfrac{1}{2},k)}}g_{{(2, \sigma)}}\psi_k\bar{k},\\
&d^{(p)}_{q+1}:=\sum_{k\in \Lambda_b}a_{(b,k)}\phi_{{(\gamma,\tfrac{1}{2},k)}}g_{{(2, \sigma)}}\psi_k\bar{\bar{k}},
\end{align*}
where $a_{(b,k)}$ and $a_{(v,k)}$ are defined in \eqref{a(v,k)a(b,k)}.

Next, we need to construct the \emph{incompressibility correctors}  to ensure that the perturbation is divergence-free.
Noting the fact that $\psi_k\bar{k}$  is zero mean and divergence free, we have
$$\psi_k\bar{k}
=\lambda^{-1}_{q+1}N^{-1}_{\Lambda}\curl\Big(-\big(\frac{\dd}{\dd x}\Psi\big)(\lambda_{q+1}N_{\Lambda} k\cdot x)k\times\bar{k}\Big):=\lambda^{-1}_{q+1}\curl F_{\bar{k}}.$$
The \emph{incompressibility correctors} $w^{(c)}_{q+1}$ and $d^{(c)}_{q+1}$ are defined as
\begin{align*}
w^{(c)}_{q+1}=&\lambda^{-1}_{q+1}\sum_{k\in \Lambda_v}{\nabla\big(a_{(v,k)}\phi_{{(\gamma,\tfrac{1}{2},k)}}g_{{(2, \sigma)}}\big)}\times F_{\bar{k}}+\lambda^{-1}_{q+1}\sum_{k\in \Lambda_b}{\nabla\big(a_{(b,k)}\phi_{{(\gamma,\tfrac{1}{2},k)}}g_{{(2, \sigma)}}\big)}\times F_{\bar{k}},
\end{align*}
and
\begin{align*}
d^{(c)}_{q+1}=\lambda^{-1}_{q+1}\sum_{k\in \Lambda_b}{\nabla\big(a_{(v,k)}\phi_{{(\gamma,\tfrac{1}{2},k)}}g_{{(2, \sigma)}}\big)} \times F_{\bar{\bar{k}}}.
\end{align*}
Owning to  $\curl(aF)=\nabla a\times F+a\curl F$, it is easy to verify that
\begin{equation}\label{wp+wc}
\begin{aligned}
w^{(p)}_{q+1}+w^{(c)}_{q+1}
=&\lambda^{-1}_{q+1}\sum_{k\in \Lambda_v}{\curl \big(a_{(v,k)}\phi_{{(\gamma,\tfrac{1}{2},k)}}g_{{(2, \sigma)}} F_{\bar{k}}\big)}\\
&+\lambda^{-1}_{q+1}\sum_{k\in \Lambda_b}{\curl \big(a_{(v,k)}\phi_{{(\gamma,\tfrac{1}{2},k)}}g_{{(2, \sigma)}} F_{\bar{k}}\big)},
\end{aligned}
\end{equation}
and
\begin{equation}\label{d^p_q+1+d^c_q+1}
d^{(p)}_{q+1}+d^{(c)}_{q+1}=\lambda^{-1}_{q+1}\sum_{k\in \Lambda_b}{\curl \big(a_{(v,k)}\phi_{{(-\gamma,\tfrac{1}{2},k)}}g_{{(2, \sigma)}} F_{\bar{\bar{k}}}\big)}.
\end{equation}
Therefore, $\div (w^{(p)}_{q+1}+w^{(c)}_{q+1})=\div (d^{(p)}_{q+1}+d^{(c)}_{q+1})=0$.

Finally, Let $h_{\sigma}(t):[0, 1]\to\R$ be defined by
\[h_{\sigma}(t)=\lambda^{-\sigma}_{q+1}\int_{0}^{\lambda^\sigma_{q+1}t} (\PP g^2_{{(2, \sigma)}})(s)\dd s,\]
 where $\PP$ is the projection onto functions with zero mean, that is $\PP f(s):=f(s)-\int_{\TTT^3}f\dd s$.

 The \emph{temporal correctors} are defined by
\begin{equation}\label{wt}
\begin{aligned}
{w^{(t)}_{q+1}}=&-\sum_{k\in \Lambda_v}\mathbb{P}_H\div( a^2_{(v,k)}\bar{k}\otimes \bar{k})h_{\sigma}(t)-\sum_{k\in \Lambda_b}\mathbb{P}_H\div(a^2_{(b,k)}(\bar{k}\otimes \bar{k}-\bar{\bar{k}}\otimes \bar{\bar{k}}))h_{\sigma}(t)\\
{d^{(t)}_{q+1}}=&-\sum_{k\in \Lambda_b}\div(a^2_{(b,k)}(\bar{k}\otimes\bar{\bar{k}} -\bar{k}\otimes \bar{\bar{k}}))h_{\sigma}(t),
\end{aligned}
\end{equation}
where $\mathbb{P}_H:={\rm Id}-\frac{\nabla\div}{\Delta}$ represents the Leray projector. Therefore, we have $\div {w^{(t)}_{q+1}}=0$. Moreover,  ${d^{(t)}_{q+1}}$ is divergence-free because $\div\div A=0$ for any anti-symmetric matrix $A$.

Moreover, we deduce by the support of $\eta_q(t)$ that
\begin{equation}\label{dwt}
(d_{p+1}, w_{q+1})(t,x)=0, \quad\forall t\in [0,2T+\tfrac{7\tau_q}{6}] \cup[ 3T-\tfrac{\tau_q}{6}, 1\big].
\end{equation}
\subsection{Estimates for the perturbation}
In this section, we establish estimates for the perturbation $(w_{q+1}, d_{q+1})$, which implies that estimates \eqref{e:vq-LpLinftyC1-}-\eqref{e:vqbq-H3} hold at $q+1$ level.
\begin{prop}[estimates on $\rho_v$ and $\rho_b$ ]\label{rho_v,rho_b}
\begin{align*}
\|(\rho_v,\rho_b)\|_{L^1_{t,x}}\le \delta_{q+1}\lambda^{-\alpha}_q.
\end{align*}
\begin{proof}By the definition of $\rho_b$ and Proposition \ref{MMM_qRRR_qL1}, we obtain that
\begin{align*}
\|\rho_b\|_{L^1_{t,x}}
=&\delta_{q+1}\lambda^{-2\alpha}_q(\int_{\{(x,t)\in \TTT^3\times [0,1]||\tfrac{|\MMM_q|}{\delta_{q+1}\lambda^{-4\alpha}_q}|\le 1\}}\chi(\langle\tfrac{\MMM_q}{\delta_{q+1}\lambda^{-4\alpha}_q}\rangle )\dd x \dd t)\\
&+\delta_{q+1}\lambda^{-2\alpha}_q(\int_{\{(x,t)\in \TTT^3\times [0,1]||\tfrac{|\MMM_q|}{\delta_{q+1}\lambda^{-4\alpha}_q}|\ge 1\}}\chi(\langle\tfrac{\MMM_q}{\delta_{q+1}\lambda^{-4\alpha}_q}\rangle )\dd x \dd t)\\
\lesssim&\delta_{q+1}\lambda^{-2\alpha}_q+\delta_{q+1}\lambda^{-2\alpha}_q
\big\|\tfrac{\MMM_q}{\delta_{q+1}\lambda^{-4\alpha}_q}\big\|_{L^1_{t,x}}\lesssim\delta_{q+1}\lambda^{-2\alpha}_q.
\end{align*}
and
\begin{align*}
\|\rho_v\|_{L^1_{t,x}}
\lesssim&\delta_{q+1}\lambda^{-2\alpha}_q+\delta_{q+1}\lambda^{-2\alpha}_q
\big(\big\|\tfrac{\RRR_q}{\delta_{q+1}\lambda^{-3\alpha}_q}\big\|_{L^1_{t,x}}+\big\|\tfrac{\rho_b}{\delta_{q+1}\lambda^{-3\alpha}_q}\big\|_{L^1_{t,x}}\big)
\lesssim\delta_{q+1}\lambda^{-\alpha}_q.
\end{align*}
\end{proof}
\end{prop}

\begin{prop}\label{a{(b,k)}a{(v,k)}}
\begin{align*}
&\|(a_{(b,k)}, a_{(v,k)})\|_{L^2_{t,x}}\lesssim\delta^{\frac{1}{2}}_{q+1}\lambda^{-\frac{\alpha}{2}}_q, \quad\,\,\,\|(a_{(b,k)}, a_{(v,k)})\|_{L^\infty_{t,x}}\le\lambda^{12}_q,\\
&\|(\partial^M_ta_{(b,k)}, \partial^M_ta_{(v,k)})\|_{L^\infty_t H^5_x}\le\lambda^{1000}_q, \, M=0,1.
\end{align*}

\end{prop}
\begin{proof}Thanks to Proposition \ref{rho_v,rho_b}, one obtains that
\begin{align*}
\|(a_{(b,k)}, a_{(v,k)})\|_{L^2_{t,x}}\lesssim&\|\rho^{\frac{1}{2}}_b\|_{L^2_{t,x}}+\|\rho^{\frac{1}{2}}_v\|_{L^2_{t,x}}
\lesssim\delta^{\frac{1}{2}}_{q+1}\lambda^{-\frac{\alpha}{2}}_q.
\end{align*}
Taking advantage of Proposition \ref{MMM_qRRR_qL1}, we get that
\begin{align*}
\|\rho_b\|_{L^\infty_{t,x}}\le & \delta_{q+1}\lambda^{-2\alpha}_q\big\|\chi\big\langle\tfrac{\MMM_q}{\delta_{q+1}\lambda^{-4\alpha}_q}\big\rangle\big\|_{L^{\infty}_{t,x}}
\lesssim\delta_{q+1}\lambda^{-2\alpha}_q(1+\tfrac{\|\MMM_q\|_{L^\infty_{t,x}}}{\delta_{q+1}\lambda^{-4\alpha}_q})\lesssim\lambda^{22+2\alpha}_{q},
\end{align*}
which implies $\|a_{(b,k)}\|_{L^\infty_{t,x}}\lesssim\|\rho\|^{\frac{1}{2}}_{L^\infty_{t,x}}\lesssim\lambda^{11+
\alpha}_q$.
Similarly,
\begin{align*}
\|a_{(v,k)}\|_{L^\infty_{t,x}}\le \|\rho^{\frac{1}{2}}_v\|_{L^\infty_{t,x}}
\lesssim\Big(\delta_{q+1}\lambda^{-2\alpha}_q(1+\frac{(\|\RRR_q\|_{L^\infty_{t,x}}+\|a^2_{(b,k)}\|_{L^\infty_{t,x}})}{\delta_{q+1}\lambda^{-3\alpha}_q})\Big)^\frac{1}{2}
\lesssim\lambda^{11+\frac{3\alpha}{2}}_q.
\end{align*}

Next, we turn to estimate $\|(a_{(b,k)}, a_{(v,k)})\|_{L^\infty_tH^5_x}$. Using the definition of $\rho_b$ and Proposition~\ref{MMM_qRRR_qL1}, we have by interpolation that
\begin{equation}\label{rho-1}
\begin{aligned}
\|\rho^{-1}_b\|_{L^\infty_t\dot H^{5}_x}&\lesssim
\delta^{-1}_{q+1}\lambda^{2\alpha}_q(1+\big\|\tfrac{\MMM_q}{\delta_{q+1}\lambda^{-4\alpha}_q}\big\|_{L^\infty_{t,x}})^{4}\big\|\tfrac{\MMM_q}{\delta_{q+1}\lambda^{-4\alpha}_q}\big\|_{L^\infty_t \dot H^{5}_x}\\
&\lesssim \delta^{-2}_{q+1}\lambda^{6\alpha}_q(1+\delta^{-4}_{q+1}\lambda^{16\alpha }_q\big\|\MMM_q\|^{4}_{L^\infty_{t,x}})\|\MMM_q\big\|_{L^\infty_t \dot H^{5}_x}\\
&{\lesssim \delta^{-6}_{q+1}\lambda^{22\alpha+173}_q},
\end{aligned}
\end{equation}
and
\begin{equation}
\begin{aligned}
\big\|\tfrac{\MMM_q}{\rho_b}\big\|_{L^\infty \dot H^{5}_x}\lesssim&\|\MMM_q\|_{L^\infty_{t,x}}\|\rho^{-1}_b\|_{L^\infty_t\dot H^{5}}+\|\MMM_q\|_{L^\infty_t\dot H^{5}_x}\|\rho^{-1}_b\|_{L^\infty_{t,x}}
\lesssim\delta^{-6}_{q+1}\lambda^{22\alpha+195}_q.
\end{aligned}
\end{equation}
Therefore, we get
\begin{align}\label{abkH4}
\big\|a_{b,k}(-\tfrac{\MMM_q}{\rho_b})\big\|_{L^\infty_t \dot H^{5}_x}\lesssim&(1+\big\|\tfrac{\MMM_q}{\rho_b}\big\|_{L^\infty_{t,x}})^4\big\|\tfrac{\MMM_q}{\rho_b}\big\|_{L^\infty_t \dot H^{5}_x}\lesssim\delta^{-10}_{q+1}\lambda^{30\alpha+283}_q.
\end{align}
An easy computation yields  that
\begin{align}\label{rhoH4}
\|\rho^{\frac{1}{2}}_b\|_{L^\infty_t \dot H^{5}_x}&\lesssim
\delta^{\frac{1}{2}}_{q+1}\lambda^{-\alpha}_q(1+\|\tfrac{\MMM_q}{\delta_{q+1}\lambda^{-4\alpha}_q}\|_{L^\infty_{t,x}})^4\|\tfrac{\MMM_q}{\delta_{q+1}\lambda^{-4\alpha}_q}\|_{L^\infty_t \dot H^{5}_x}\lesssim
\delta^{-5}_{q+1}\lambda^{19\alpha+173}_q.
\end{align}
Collecting the above two estimates, we have
\begin{align*}
\|a_{(b,k)}\|_{L^\infty_t \dot H^{5}_x}\lesssim&\|\rho^{\frac{1}{2}}_b\|_{L^\infty_{t,x}}\|a_{b,k}(-\tfrac{\MMM_q}{\rho_b})\|_{L^\infty_t \dot H^{5}_x}+\|\rho^{\frac{1}{2}}_b\|_{L^\infty_t \dot H^{5}_x}\|a_{b,k}(-\tfrac{\MMM_q}{\rho_b})\|_{L^\infty_{t,x}}
\le\lambda^{300}_q.
\end{align*}
In the same way as deriving bound on $\|a_{(b,k)}\|_{L^\infty_t \dot H^{5}_x}$, one obtains  that
\begin{align*}
\|a_{(v,k)}\|_{L^\infty_t \dot H^{5}_x} \le \lambda^{700}_q.
\end{align*}
Finally, we aim to bound $\|(\partial_ta_{(b,k)}, \partial_ta_{(b,k)})\|_{L^\infty_tH^5_x}$. By the definition of $a_{(b,k)}$, we get
\begin{equation}\label{abktH3}
\begin{aligned}
\|\partial_ta_{(b,k)}\|_{L^\infty_t \dot H^{5}_x}\lesssim&\|\partial_t(\eta_q\rho^{\frac{1}{2}}_b)\|_{L^\infty_{t,x}}\|a_{b,k}(-\tfrac{\MMM_q}{\rho_b})\|_{L^\infty_t \dot H^{5}_x}+\|\partial_t(\eta_q\rho^{\frac{1}{2}}_b)\|_{L^\infty_t \dot H^{5}_x}\|a_{b,k}(-\tfrac{\MMM_q}{\rho_b})\|_{L^\infty_{t,x}}\\
&+\|\eta_q\rho^{\frac{1}{2}}_b\|_{L^\infty_{t,x}}\|\partial_t(a_{b,k}(-\tfrac{\MMM_q}{\rho_b}))\|_{L^\infty_t \dot H^{5}_x}+\|\eta_q\rho^{\frac{1}{2}}_b\|_{L^\infty_t \dot H^{5}_x}\|\partial_t(a_{b,k}(-\tfrac{\MMM_q}{\rho_b}))\|_{L^\infty_{t,x}}.
\end{aligned}
\end{equation}
By a similar method as in estimating \eqref{rho-1}--\eqref{abkH4}, we infer by Proposition \ref{MMM_qRRR_qL1} that
\begin{align*}
&\|\partial_t\rho^{\frac{1}{2}}_b\|_{L^\infty_{t,x}}\lesssim \delta^{-\frac{1}{2}}_{q+1}\lambda^{3\alpha}_q\|\partial_t\MMM_q\|_{L^\infty_{t,x}}\le \lambda^{60}_{q},
\quad\|\partial_t\rho^{\frac{1}{2}}_b\|_{L^\infty_t \dot H^5_x}\le \lambda^{260}_{q},\\
&\|\partial_t(a_{b,k}(-\tfrac{\MMM_q}{\rho_b}))\|_{L^\infty_{t,x}}\lesssim\lambda^{90}_q,
\quad\|\partial_t(a_{b,k}(-\tfrac{\MMM_q}{\rho_b}))\|_{L^\infty_t\dot H^5_x}\le\lambda^{450}_q.
\end{align*}
Plugging the above two estimates into \eqref{abktH3}, we obtain that
\begin{align*}
\|\partial_ta_{(b,k)}\|_{L^\infty_t  H^{5}_x}\le\lambda^{520}_q.
\end{align*}
In the same way, a simple computation shows that $\|\partial_ta_{(v,k)}\|_{L^\infty_t  H^{5}_x}\le\lambda^{990}_q.$

\end{proof}
\begin{prop}\label{es-wdL2}
\begin{align*}
&\|(d^{(p)}_{q+1}, w^{(p)}_{q+1})\|_{L^2_{t,x}}\lesssim \delta^{\frac{1}{2}}_{q+1}\lambda^{-\frac{\alpha}{2}}_q,\\
&\|(d^{(c)}_{q+1}, w^{(c)}_{q+1})\|_{L^2_{t,x}}\le \lambda^{C_0}_q \lambda^{\gamma-\frac{1}{2}}_{q+1},\\
&\|(d^{(t)}_{q+1}, w^{(t)}_{q+1})\|_{L^2_{t,x}}\le\lambda^{C_0}_q\lambda^{-\iota}_{q+1}.
\end{align*}
\end{prop}
\begin{proof}
By Proposition \ref{guji1}, Proposition \ref{a{(b,k)}a{(v,k)}} and Lemma \ref{Holder}, thanks to $\iota<\tfrac{1}{2}$ and condition~\eqref{gamma}, we have
\begin{align*}
&\|d^{(p)}_{q+1}\|_{L^2_{t,x}}+\|w^{(p)}_{q+1}\|_{L^2_{t,x}}\\
\lesssim & \|(a_{(b,k)},a_{(v,k)})\|_{L^2_{t,x}}\|g_{(2,\sigma)}\|_{L^2_t}\|\phi_{(\gamma, \frac{1}{2},k)}\|_{L^2_x}+\lambda^{-\frac{\sigma}{2}}_{q+1}\|(a_{(b,k)},a_{(v,k)})\|_{C^1_{t,x}}\|g_{(2,\sigma)}\|_{L^2_t}\|\phi_{(\gamma, \frac{1}{2},k)}\|_{L^2_x}\\
\lesssim & \delta^{\frac{1}{2}}_{q+1}\lambda^{-\frac{\alpha}{2}}_q+\lambda^{-\frac{\sigma}{2}}_{q+1}\lambda^{500}_q\lesssim \delta^{\frac{1}{2}}_{q+1}\lambda^{-\frac{\alpha}{2}}_q.
\end{align*}
Taking advantage  of Proposition \ref{guji1} and Proposition \ref{a{(b,k)}a{(v,k)}}, one obtains that
\begin{align*}
\|(d^{(c)}_{q+1},w^{(c)}_{q+1})\|_{L^2_{t,x}}&\lesssim \lambda^{-1}_{q+1}\|(a_{(b,k)},a_{(v,k)})\|_{L^\infty_t W^{1,\infty}_x}\|\phi_{(\gamma,\frac{1}{2},k)}\|_{H^1_x}\|g_{(2,\sigma)}\|_{L^2_t}\lesssim\lambda^{1000}_q\lambda^{\gamma-\frac{1}{2}}_{q+1}.
\end{align*}
Similarly, we have
\begin{align*}
\|(d^{(t)}_{q+1}, w^{(t)}_{q+1})\|_{L^2_{t,x}}\lesssim& \|(a^2_{(v,k)},a^2_{(b,k)})\|_{L^2_tH^1_x}
\|h_{\sigma}(t)\|_{L^\infty_t}
\lesssim \lambda^{1012}_q\|h_{\sigma}(t)\|_{L^\infty_t}.
\end{align*}
Note that $\g^2$ is a $1-$ periodic function, we obtain
\begin{align}\label{ginfty}
\big\|h_{\sigma}(t)\big\|_{L^\infty_t}=\lambda^{-\sigma}_{q+1}\big\|\int_{\lfloor\lambda^\sigma_{q+1}t\rfloor}^{\lambda^\sigma_{q+1}t} (\PP \g^2)(s)\dd s\big\|_{L^\infty_t}\le 2\lambda^{-\sigma}_{q+1}\|\g^2\|_{L^1[0,1]}= 2\lambda^{-\sigma}_{q+1}.
\end{align}
In conclusion, we complete the proof of Proposition \ref{es-wdL2} for $C_0=2^{12}$.
\end{proof}
\begin{prop}\label{es-wdLpLinfty}Let $1\le p\le\infty$ and $\e>0$, we have
\begin{align}
&\|(d^{(p)}_{q+1}, d^{(c)}_{q+1}, d^{(t)}_{q+1}, w^{(p)}_{q+1}, w^{(c)}_{q+1}, w^{(t)}_{q+1})\|_{L^p_t L^\infty_x}\le
\lambda^{C_0}_q(\lambda^{1+\frac{\gamma}{2}-\frac{2}{p}}_{q+1}+\lambda^{-\sigma}_{q+1}),\label{dLp}\\
&\|(d^{(p)}_{q+1}, d^{(c)}_{q+1}, d^{(t)}_{q+1},  w^{(p)}_{q+1}, w^{(c)}_{q+1}, w^{(t)}_{q+1})\|_{L^\infty_x H^{3}_x}\lesssim{\lambda^{C_0}_q}\lambda^{4+\frac{\gamma}{2}}_{q+1},\label{dH3}\\
&\|(d_{q+1}, w_{q+1})\|_{L^1_t C^{1-\e}_x}\le\lambda^{C_0}_q(\lambda^{\frac{\gamma}{2}-\e}_{q+1}+\lambda^{-\sigma}_{q+1})\label{dC1}.
\end{align}
\end{prop}
\begin{proof}By the definition of $(d^{(p)}_{q+1}, d^{(c)}_{q+1}, w^{(p)}_{q+1}, w^{(c)}_{q+1})$, we get by Proposition \ref{guji1} and Proposition \ref{a{(b,k)}a{(v,k)}} that
\begin{equation}\label{dpwpLpLinfty}
\begin{aligned}
&\|d^{(p)}_{q+1}\|_{L^p_tL^\infty_x}+\|d^{(c)}_{q+1}\|_{L^p_tL^\infty_x}+ \|w^{(p)}_{q+1}\|_{L^p_tL^\infty_x}+\|w^{(c)}_{q+1}\|_{L^p_tL^\infty_x}\\
\lesssim & \|(a_{(b,k)},a_{(v,k)})\|_{L^\infty_{t,x}}\|\p\|_{L^\infty_x}\|\g\|_{L^p_t}\|\psi_k\|_{L^\infty_x}\\
&+\lambda^{-1}_{q+1}\|(a_{(b,k)},a_{(v,k)})\|_{L^\infty_t W^{1,\infty}_x}\|\p\|_{W^{1,\infty}_x}\|\g\|_{L^p_t}\|(F_{\bar{k}}, F_{\bar{\bar{k}}})\|_{L^\infty_x}\\
\lesssim& \lambda^{12}_{q}\lambda^{1+\frac{\gamma}{2}-\frac{2}{p}}_{q+1}+\lambda^{1000}_q\lambda^{\frac{3}{2}\gamma-\frac{1}{2}+\frac{2}{p}}_{q+1}
\lesssim\lambda^{1000}_q\lambda^{1+\frac{\gamma}{2}-\frac{2}{p}}_{q+1}.
\end{aligned}
\end{equation}
Since $H^2(\TTT^3)\hookrightarrow L^\infty (\TTT^3)$ and Leray projetor $\mathbb{P}_{H}$ is bounded on $L^2\to L^2$, we have by Proposition~\ref{a{(b,k)}a{(v,k)}} and  \eqref{ginfty} that
\begin{align}\label{dpwpt}
&\|(d^{(t)}_{q+1}, w^{(t)}_{q+1})\|_{L^p_tL^\infty_x}
\lesssim \lambda^{-\sigma}_{q+1}\|(a^2_{(b,k)}, a^2_{(v,k)})\|_{L^\infty_tH^3_x}\lesssim \lambda^{1012}_q\lambda^{-\sigma}_{q+1}.
\end{align}
Collecting the above two inequalities, for $C_0=2^{12}$, we deduce \eqref{dLp}.

In the same way as estimating \eqref{dpwpLpLinfty}, one gets
by Proposition \ref{a{(b,k)}a{(v,k)}} that
\begin{align*}
&\|(d^{(p)}_{q+1}, d^{(c)}_{q+1}, d^{(t)}_{q+1},  w^{(p)}_{q+1}, w^{(c)}_{q+1}, w^{(t)}_{q+1})\|_{L^\infty_tH^3_x}\\
\lesssim & \|(a_{(b,k)},a_{(v,k)})\|_{L^\infty_t H^3_x}\|\g\|_{L^\infty_t}(\|\psi_k\|_{L^\infty_{t,x}}\|\p\|_{H^3_x}+\|\psi_k\|_{H^3_x}\|\p\|_{L^\infty_{t,x}})\\
&+\lambda^{-1}_{q+1} \|(a_{(b,k)},a_{(v,k)})\|_{L^\infty_t H^4_x}\|\g\|_{L^\infty_t}\|\p\|_{L^\infty_x}\|(F_{\bar{k}}, F_{\bar{\bar{k}}})\|_{W^{4,\infty}_x}\\
&+\lambda^{-1}_{q+1} \|(a_{(b,k)},a_{(v,k)})\|_{L^\infty_t H^4_x}\|\g\|_{L^\infty_t}\|\p\|_{W^{4,\infty}_x}\|(F_{\bar{k}}, F_{\bar{\bar{k}}})\|_{L^\infty_x}\\
&+\lambda^{-\sigma}_{q+1}\|(a^2_{(b,k)}, a^2_{(v,k)})\|_{L^\infty_t H^4_x}\\
\lesssim& {\lambda^{C_0}_q}\lambda^{4+\frac{\gamma}{2}}_{q+1},
\end{align*}
which implies \eqref{dH3}. One gets by \eqref{wp+wc}, \eqref{d^p_q+1+d^c_q+1} Proposition~\ref{guji1} and Proposition \ref{a{(b,k)}a{(v,k)}} that
\begin{align*}
&\|d^{(p)}_{q+1}+d^{(c)}_{q+1}\|_{L^1_t C^1_x}+\|w^{(p)}_{q+1}+w^{(c)}_{q+1}\|_{L^1_t C^1_x}\\
\lesssim & \lambda^{-1}_{q+1}\|(a_{(b,k)},a_{(v,k)})\|_{L^\infty_t W^{2,\infty}_x}\|\p\|_{W^{2,\infty}_x}\|\g\|_{L^1_t}\|(F_{\bar{k}}, F_{\bar{\bar{k}}})\|_{L^\infty_x}\\
&+\lambda^{-1}_{q+1}\|(a_{(b,k)},a_{(v,k)})\|_{L^\infty_{t ,x}}\|\p\|_{L^\infty_x}\|\g\|_{L^1_t}\|(F_{\bar{k}}, F_{\bar{\bar{k}}})\|_{W^{2,\infty}_x}\\
\lesssim& \lambda^{1000}_{q}\lambda^{\frac{\gamma}{2}-1}_{q+1}+\lambda^{12}_{q}\lambda^{\frac{\gamma}{2}}_{q+1}
\le\lambda^{1000}_q\lambda^{\frac{\gamma}{2}}_{q+1}.
\end{align*}
In a similar way as  in the proof of \eqref{dpwpt}, we have
\begin{align*}
\|(d^{(t)}_{q+1}, w^{(t)}_{q+1})\|_{L^1_t C^1_x}\lesssim &\lambda^{-\sigma}_{q+1}\|(a^2_{(b,k)}, a^2_{(v,k)})\|_{L^\infty_t H^3_x}\le\lambda^{1012}_q\lambda^{-\iota}_{q+1}.
\end{align*}
Using interpolation, we obtain
\begin{align*}
\|(d_{q+1}, w_{q+1})\|_{L^1_t C^{1-\e}_x}\lesssim &\|(d^{(p)}_{q+1}+d^{(c)}_{q+1}, w^{(p)}_{q+1}+w^{(c)}_{q+1})\|^{1-\e}_{L^1_t C^{1}_x}\|(d^{(p)}_{q+1}+d^{(c)}_{q+1}, w^{(p)}_{q+1}+w^{(c)}_{q+1})\|^{\e}_{L^1_t L^\infty_x}\\
&+\|(d^{(t)}_{q+1}, w^{(t)}_{q+1})\|^{1-\e}_{L^1_t C^{1}_x}\|(d^{(t)}_{q+1}, w^{(t)}_{q+1})\|^{\e}_{L^1_t L^\infty_x}\\
\le&\lambda^{C_0}_q\lambda^{\frac{\gamma}{2}-\e}_{q+1}+\lambda^{C_0}_q\lambda^{-\sigma}_{q+1}.
\end{align*}
Therefore, we complete the proof of Proposition \ref{es-wdLpLinfty}.
\end{proof}
\section{Convex integration step: estimates for Reynolds and magnetic stresses }
Let $(v_{q+1}, b_{q+1})=(\vv_q+w_{q+1}, \bb_{q}+d_{q+1})$. One verifies that $(v_{q+1}, b_{q+1}, p_{q+1}, \RR_{q+1}, \MM_{q+1})$ satisfies \eqref{e:subsol-euler} with replacing $q$ by $q+1$, where
\begin{align*}
\RR_{q+1}
=&\underbrace{\mathcal{R}[\big(\del_t w_{q+1}-\Delta w_{q+1}\big)+\div(w_{q+1}\otimes \vv_q+\vv_q\otimes w_{q+1}-\bb_q\otimes d_{q+1}-d_{q+1}\otimes \bb_q)]}_{\Rlin}\\
&+\underbrace{\mathcal{R}[\div\big(\RRR_q+(w_{q+1}\otimes w_{q+1})-(d_{q+1}\otimes d_{q+1})\big)-\nabla p_v]}_{\Rosc},\\
\MM_{q+1}
=&\underbrace{\mathcal{R}_a[\big(\del_t w_{q+1}-\Delta w_{q+1}\big)+\div(w_{q+1}\otimes \bb_q+\vv_q\otimes d_{q+1}
-\bb_q\otimes w_{q+1}-d_{q+1}\otimes \vv_q)]}_{\Mlin}\\
&+\underbrace{\mathcal{R}_a[\div\big(\MMM_q+w_{q+1}\otimes d_{q+1}-d_{q+1}\otimes w_{q+1}\big)]}_{\Mosc},
\end{align*}
and $p_{q+1}=\pp_{q}-p_v,$
where
\begin{equation}\label{pv}
\begin{aligned}
p_v=&\eta^2_q\rho_v-\sum_{k\in \Lambda_v}\frac{\div\div}{\Delta}a^2_{(v,k)}(\bar{k}\otimes\bar{k})\PP(\g^2)\\
&-\sum_{k\in \Lambda_b}\frac{\div\div}{\Delta}a^2_{(b,k)}(\bar{k}\otimes\bar{k}-\bar{\bar{k}}\otimes\bar{\bar{k}})\PP(\g^2).
\end{aligned}
\end{equation}
\subsection{Estimates for the magnetic stress}
\begin{prop}\label{Mlin}
\begin{align}
&\|\Mlin-\mathcal{R}_a\del_t d^{(t)}_{q+1}\|_{L^1_{t,x}}\lesssim \lambda^{2C_0}_q(\lambda^{\sigma-\frac{\gamma}{4}}_{q+1}+\lambda^{-\frac{\gamma}{6}}_{q+1}+\lambda^{-\sigma}_{q+1}),\label{MlinL1}\\
&{\|\Mlin-\mathcal{R}_a\del_t d^{(t)}_{q+1}\|_{L^\infty_t H^3_x}\lesssim \lambda^{C_0}_{q}\lambda^{5+\sigma+\frac{\gamma}{2}}_{q+1}}\label{MlinH3}.
\end{align}
\end{prop}
\begin{proof}By the definition of $\Mlin$, one has
\begin{align*}
\Mlin-\mathcal{R}_a\del_t d^{(t)}_{q+1}
=&\mathcal{R}_a\big(\del_t d^{(p)}_{q+1}+\del_t d^{(c)}_{q+1}-\Delta d_{q+1}\big)\\
&+\mathcal{R}_a(\div(w_{q+1}\otimes \bb_q-\bb_q\otimes w_{q+1}+\vv_q\otimes d_{q+1}-d_{q+1}\otimes \vv_q)).
\end{align*}
Since Calderon-Zygmund operators are bounded on $L^p\to L^p$, $\forall 1<p<\infty$, we obtain that
\begin{equation}\label{es-Mlin-1}
\begin{aligned}
\|\mathcal{R}_a\del_t (d^{(p)}_{q+1}+d^{(c)}_{q+1})\|_{L^1_{t,x}}
\lesssim&\lambda^{-1}_{q+1}\|\del_t(a_{(b,k)}\p\g) F_{\bar{\bar{k}}})\|_{L^1_tL^{\frac{3}{2}}_x}\\
\lesssim&\lambda^{-1}_{q+1}\|a_{(b,k)}\|_{C^1_{t,x}}\|\p\|_{L^{\frac{3}{2}}_x}\|\g\|_{W^{1,1}_t}\|F_{\bar{\bar{k}}}\|_{L^\infty_x}\\
\le&\lambda^{C_0}_q \lambda^{\sigma-\frac{\gamma}{4}}_{q+1},
\end{aligned}
\end{equation}
and
\begin{equation}\label{es-Mlin-2}
\begin{aligned}
\|\mathcal{R}_a\Delta d_{q+1}\|_{L^1_{t,x}}
\lesssim&\|\nabla (d ^{(p)}_{q+1}+d^{(c)}_{q+1})\|_{L^1_tL^{\frac{3}{2}}_x}+\|\nabla d^{(t)}_{q+1}\|_{L^1_tL^{\frac{3}{2}}_x}\\
\lesssim&\lambda^{-1}_{q+1}\|a_{(b,k)}\|_{L^\infty_t W^{2,\infty}_x}\|\g\|_{L^1_t}\|\p\|_{L^\frac{3}{2}_x}\|F_{\bar{\bar{k}}}\|_{W^{2,\infty}_x}\\
&+\lambda^{-1}_{q+1}\|a_{(b,k)}\|_{L^\infty _t W^{2,\infty}_x}\|\g\|_{L^1_t}\|\p\|_{W^{2,\frac{3}{2}}_x}\|F_{\bar{\bar{k}}}\|_{L^\infty_x}\\
&+\lambda^{-\sigma}_{q+1}\|a^2_{(b,k)}\|_{L^\infty_tW^{2,\infty}_x}\\
\lesssim&\lambda^{C_0}_{q}\lambda^{-\frac{\gamma}{6}}_{q+1}+\lambda^{C_0}_{q}\lambda^{-\sigma}_{q+1}.
\end{aligned}
\end{equation}
Using Proposition \ref{estimate-vvq} and Proposition \ref{es-wdLpLinfty}, one deduces that
\begin{equation}\label{es-Mlin-3}
\begin{aligned}
&\|\mathcal{R}_a\div(w_{q+1}\otimes \bb_q-\bb_q\otimes w_{q+1}+\vv_q\otimes d_{q+1}-d_{q+1}\otimes \vv_q) \|_{L^1_{t,x}}\\
\lesssim&\|w_{q+1}\otimes \bb_q-\bb_q\otimes w_{q+1}+\vv_q\otimes d_{q+1}-d_{q+1}\otimes \vv_q \|_{L^1_t L^{\frac{3}{2}}_x}\\
\lesssim&\|w_{q+1}\|_{L^1_tL^\infty_x}\|\bb_q\|_{L^\infty_{t,x}}+\|d_{q+1}\|_{L^1_tL^\infty_x}\|\vv_q\|_{L^\infty_{t,x}}\lesssim \lambda^{2C_0}_q\lambda^{\frac{\gamma}{2}-1}_{q+1}+\lambda^{2C_0}_q\lambda^{{-\sigma}}_{q+1}.
\end{aligned}
\end{equation}
Collecting \eqref{es-Mlin-1}--\eqref{es-Mlin-3} shows estimate \eqref{MlinL1}. Now we turn to estimate $\|\Mlin-\mathcal{R}_a\del_t d^{(t)}_{q+1}\|_{L^\infty_t H^3_x}$. Thanks to the equality \eqref{d^p_q+1+d^c_q+1}, we have by Proposition~\ref{guji1} and Proposition~\ref{a{(b,k)}a{(v,k)}} that
\begin{equation}\label{es-Mlin-H3}
\begin{aligned}
&\|\mathcal{R}_a\del_t (d^{(p)}_{q+1}+d^{(c)}_{q+1})\|_{L^\infty_t H^3_x}\\
\lesssim&\lambda^{-1}_{q+1}(\|\del_t(a_{(b,k)\g}\p) \|_{L^\infty_{t,x}}\|F_{\bar{\bar{k}}}\|_{H^3_x}+\|\del_t(a_{(b,k)\g}\p)  \|_{L^\infty_t H^3_x}\|F_{\bar{\bar{k}}}\|_{L^\infty_x})\\
\lesssim&\lambda^{-1}_{q+1}\|\g\|_{C^1_t}(\lambda^3_{q+1}\|a_{(b,k)}\|_{C^1_{t,x}}\|\p\|_{L^\infty_x}+\|(\partial_ta_{(b,k)},a_{(b,k)})\|_{L^\infty_t H^3_x}\|\p\|_{H^3_x})\\
\lesssim&\lambda^{C_0}_{q}\lambda^{5+\sigma+\frac{\gamma}{2}}_{q+1}.
\end{aligned}
\end{equation}
Similarly, we can deduce that
\begin{equation}\label{es-Mlin-H32}
\begin{aligned}
\|\mathcal{R}_a\Delta d_{q+1}\|_{L^\infty_t H^3_x}
\lesssim&\|(d ^{(p)}_{q+1}+d^{(c)}_{q+1})\|_{L^\infty_t H^4_x}+\| d^{(t)}_{q+1}\|_{L^\infty_t H^4_x}\\
\lesssim&\lambda^{-1}_{q+1}\|a_{(b,k)}\|_{L^\infty_t H^4_x}\|\p\|_{H^4_x}\|\g\|_{L^\infty_t}\|F_{\bar{k}}\|_{L^\infty}\\
&+\lambda^{-1}_{q+1}\|a_{(b,k)}\|_{L^\infty_{t,x}}\|\p\|_{L^\infty_x}\|\g\|_{L^\infty_t}\|F_{\bar{k}}\|_{H^4_x}
+\lambda^{-\sigma}_{q+1}\|a^2_{(b,k)}\|_{L^\infty_tH^5_x}\\
\le&\lambda^{C_0}_q\lambda^{4+\frac{\gamma}{2}}_{q+1}.
\end{aligned}
\end{equation}
Combining Proposition \ref{estimate-vvq} with Proposition \ref{es-wdLpLinfty}, we get
\begin{equation}\label{dlinH3}
\begin{aligned}
&{\|\mathcal{R}_a\div(w_{q+1}\otimes \bb_q-\bb_q\otimes w_{q+1}+\vv_q\otimes d_{q+1}-d_{q+1}\otimes \vv_q) \|_{L^\infty_t H^3_x}}\\
\lesssim&\|(w_{q+1},d_{q+1})\|_{L^\infty_t H^3_x}\|(\bb_q,\vv_q)\|_{L^\infty_{t,x}}+\|(w_{q+1},d_{q+1})\|_{L^\infty_{t,x}}\|(\bb_q,\vv_q)\|_{L^\infty_t H^3_x}
\le\lambda^{2C_0}_q\lambda^{4+\frac{\gamma}{2}}_{q+1}.
\end{aligned}
\end{equation}
The above inequality combined with \eqref{es-Mlin-H3} and \eqref{es-Mlin-H32} implies \eqref{MlinH3}.
\end{proof}
\begin{prop}
\begin{align}
&\|\Mosc+\mathcal{R}_a \partial_t d^{(t)}_{q+1}\|_{L^1_{t,x}}\lesssim \lambda^{C_0}_q(\lambda^{-\frac{\gamma}{6}}_{q+1}+\lambda^{\gamma-\frac{1}{2}}_{q+1}+\lambda^{-\sigma}_{q+1}),\label{MoscL1}\\
&\|\Mosc+\mathcal{R}_a \partial_t d^{(t)}_{q+1}\|_{L^\infty_t H^3_x}\lesssim \lambda^{2C_0}_{q}\lambda^{5+\gamma}_{q+1}.\label{MoscH3}
\end{align}
\end{prop}
\begin{proof}To estimate the magnetic oscillation term and temporal flow term, we rewrite it as follows:
\begin{equation}\label{Mosc}
\begin{aligned}
\Mosc+\mathcal{R}_a \partial_t d^{(t)}_{q+1}=&\mathcal{R}_a\big[\div(\MMM_q+(w^{(p)}_{q+1}\otimes d^{(p)}_{q+1})-(d^{(p)}_{q+1}\otimes w^{(p)}_{q+1}))+ \partial_t d^{(t)}_{q+1}\big]\\
&+\mathcal{R}_a\div\big[(w^{(c)}_{q+1}+w^{(t)}_{q+1})\otimes d_{q+1}+w^{(p)}_{q+1}\otimes (d^{(c)}_{q+1}+d^{(t)}_{q+1})\\
&-(d^{(c)}_{q+1}+d^{(t)}_{q+1})\otimes w_{q+1}+d^{(p)}_{q+1}\otimes (w^{(c)}_{q+1}+w^{(t)}_{q+1})\big]\\
:=&O^b_1+O^b_2.
\end{aligned}
\end{equation}
For $O^b_1$, we obtain by the definition of $(w^{(p)}_{q+1}, d^{(p)}_{q+1})$ that
\begin{equation}\label{wpdp-dpwp}
\begin{aligned}
O^b_1
=&\mathcal{R}_a\div[\MMM_q
+\sum_{k\in\Lambda_{b}}a^2_{(b,k)}\p^2\psi^2_k\g^2(\bar k\otimes\bar{\bar{k}}-\bar{\bar{k}}\otimes\bar k)]+\mathcal{R}_a\partial_t d^{(t)}_{q+1}\\
&+\mathcal{R}_a\div[\sum_{k\in \Lambda_v,k'\in \Lambda_b}a_{(v,k)}a_{(b,k')}\p\pd\psi_k\psi_{k'}\g^2(\bar{k}\otimes \bar{k'}-\bar{k'}\otimes \bar{k})]\\
&+\mathcal{R}_a\div[\sum_{k\in\neq k'\in \Lambda_b}a_{(v,k)}a_{(b,k')}\p\pd\psi_k\psi_{k'}\g^2(\bar{k}\otimes \bar{k'}-\bar{k'}\otimes \bar{k})]\\
:=&O^b_{1,1}+O^b_{1,2}+O^b_{1,3}.
\end{aligned}
\end{equation}
Noting the fact that $k\neq k'$, we obtain by Proposition \ref{guji1} that
\begin{equation}\label{Moscd}
\begin{aligned}
\big\|O^b_{1,2}\|_{L^1_{t,x}}+\big\|O^b_{1,3}\|_{L^1_{t,x}}
\lesssim &{\lambda^{24}_q}\|\p\pd\|_{L^{\frac{3}{2}}_x}\|\g^2\|_{L^1_t}
\lesssim {\lambda^{24}_q}\lambda^{-\frac{\gamma}{6}}_{\lambda_{q+1}}.
\end{aligned}
\end{equation}
In terms of $O^b_{1,1}$, we decompose it into four parts:
\begin{equation}\label{Moscdecom}
\begin{aligned}
O^b_{1,1}
=&\mathcal{R}_a\div\big[\MMM_q
+\sum_{k\in\Lambda_{b}}a^2_{(b,k)}(\bar k\otimes\bar{\bar{k}}-\bar{\bar{k}}\otimes\bar k)\big]\\
&+\mathcal{R}_a\div\big[\sum_{k\in\Lambda_{b}}a^2_{(b,k)}(\PP\psi^2_k)\p^2\g^2(\bar k\otimes\bar{\bar{k}}-\bar{\bar{k}}\otimes\bar k)\big]\\
&+\mathcal{R}_a\div\big[\sum_{k\in\Lambda_{b}}a^2_{(b,k)}(\PP\p^2
)\g^2(\bar k\otimes\bar{\bar{k}}-\bar{\bar{k}}\otimes\bar k)\big]\\
&+\mathcal{R}_a\div\big[\sum_{k\in\Lambda_{b}}a^2_{(b,k)}(\PP\g^2)(\bar k\otimes\bar{\bar{k}}-\bar{\bar{k}}\otimes\bar k)\big]+\mathcal{R}_a\partial_t d^{(t)}_{q+1}.
\end{aligned}
\end{equation}
By Lemma \ref{first L}, thanks to $\eta_q(t)\equiv 1$ for $t\in \text{supp} \,\MMM_q$, we obtain that
\begin{align}\label{M1}
&\MMM_q
+\sum_{k\in\Lambda_{b}}a^2_{(b,k)}(\bar k\otimes\bar{\bar{k}}-\bar{\bar{k}}\otimes\bar k)
=\MMM_q-\eta^2_q\MMM_q=0.
\end{align}
With respect to the second term on the right-hand side of \eqref{Moscdecom}, due to $k\perp \bar{k}\perp \bar{\bar{k}}$, one deduces by Lemma
\ref{holder2} that
\begin{equation}\label{M2}
\begin{aligned}
&\Big\|\mathcal{R}_a\div\big[\sum_{k\in\Lambda_{b}}a^2_{(b,k)}(\PP\psi^2_k)\p^2\g^2(\bar k\otimes\bar{\bar{k}}-\bar{\bar{k}}\otimes\bar k)\big]\Big\|_{L^1_{t,x}}\\
=&\Big\|\mathcal{R}_a\sum_{k\in\Lambda_{b}}\div(a^2_{(b,k)}\bar k\otimes\bar{\bar{k}}-\bar{\bar{k}}\otimes\bar k)(\PP\psi^2_k)\p^2\g^2\Big\|_{L^1_{t,x}}\\
\lesssim&\lambda^{-1}_{q+1}\|a^2_{(b,k)}\|_{L^\infty_t W^{2,\infty}_x}\|\p^2\|_{L^\infty_x}\|\g^2\|_{L^1_t}\le \lambda^{C_0}_q\lambda^{\gamma-1}_{q+1}.
\end{aligned}
\end{equation}
In the same argument as deriving the above estimate, we can bound the second term on the right-hand side of \eqref{Moscdecom} by
\begin{equation}\label{M3}
\begin{aligned}
&\Big\|\mathcal{R}_a\div\big[\sum_{k\in\Lambda_{b}}a^2_{(b,k)}(\PP\p^2
)\g^2(\bar k\otimes\bar{\bar{k}}-\bar{\bar{k}}\otimes\bar k)\big]\Big\|_{L^1_{t,x}}\\
=&\Big\|\mathcal{R}_a\sum_{k\in\Lambda_{b}}\div(a^2_{(b,k)}\bar k\otimes\bar{\bar{k}}-\bar{\bar{k}}\otimes\bar k)(\PP\p^2
)\g^2\Big\|_{L^1_{t,x}}\\
\lesssim&\lambda^{-\frac{1}{2}}_{q+1}\|a^2_{(b,k)}\|_{L^\infty_t W^{2,\infty}_x}\|\g^2\|_{L^1_t}\le \lambda^{C_0}_q\lambda^{-\frac{1}{2}}_{q+1}.
\end{aligned}
\end{equation}
According to the definition of $d^{(t)}_{q+1}$, thanks to  $k\perp \bar{k}\perp \bar{\bar{k}}$,  we have
\begin{equation}\label{M4}
\begin{aligned}
&\mathcal{R}_a\div\big[\sum_{k\in\Lambda_{b}}a^2_{(b,k)}(\PP\g^2)(\bar k\otimes\bar{\bar{k}}-\bar{\bar{k}}\otimes\bar k)\big]+\mathcal{R}_a\partial_t d^{(t)}_{q+1}\\
=&\mathcal{R}_a\sum_{k\in\Lambda_{b}}\div\big[a^2_{(b,k)}(\bar k\otimes\bar{\bar{k}}-\bar{\bar{k}}\otimes\bar k)\big](\PP\g^2)+\mathcal{R}_a\partial_t d^{(t)}_{q+1}=0.
\end{aligned}
\end{equation}
Applying estimate \eqref{M1}--\eqref{M4} to \eqref{Moscdecom}, we can infer that
\begin{align}\label{MoscmL1}
\|O^b_{1,1}\|_{L^{1}_{t,x}}\lesssim\lambda^{C_0}_q\lambda^{-\frac{1}{2}}_{q+1}.
\end{align}
Now we turn to estimate $O^b_2$ in \eqref{Mosc}. By Proposition~\ref{es-wdL2} and Proposition \ref{es-wdLpLinfty}, we have
\begin{equation}\label{dlinL1}
\begin{aligned}
\|O^b_2\|_{L^1_{t,x}}
\lesssim &\|(w ^{(c)}_{q+1}, w^{(t)}_{q+1}, d^{(c)}_{q+1}, d^{(t)}_{q+1})\|_{L^2_{t,x}}\|(d ^{(p)}_{q+1}, d_{q+1}, w^{(p)}_{q+1}, w_{q+1})\|_{L^2_{t,x}}\\
\lesssim &\lambda^{C_0}_q\lambda^{\gamma-\frac{1}{2}}_{q+1}
+\lambda^{C_0}_q\lambda^{-\iota}_{q+1}.
\end{aligned}
\end{equation}
Combining the above inequality with \eqref{Moscd} and  \eqref{MoscmL1}, we complete the proof of \eqref{MoscL1}.

Furthermore, with aid of Proposition \ref{a{(b,k)}a{(v,k)}} and Proposition \ref{guji1}, one infer from \eqref{Moscd}, \eqref{M2} and \eqref{M3} that
\begin{equation}\label{MoscH31}
\begin{aligned}
\big\|O^b_{1,2}\|_{L^\infty_tH^3_x}
\lesssim &\|(a_{(v,k)}a_{(b,k')})\|_{L^\infty_tH^3_x}\|\p\pd\|_{L^\infty_x}\|\psi_k\psi_{k'}\|_{H^3_x}
\|\g^2\|_{L^\infty_t}\\
&+\|(a_{(v,k)}a_{(b,k')})\|_{L^\infty_tH^3_x}\|\psi_k\psi_{k'}\|_{L^\infty_x}\|\p\pd\|_{H^3_x}
\|\g^2\|_{L^\infty_t}\\
&\lesssim\lambda^{C_0}_q\lambda^{5+\gamma}_{q+1},
\end{aligned}
\end{equation}
and
\begin{equation}\label{MoscH32}
\begin{aligned}
\big\|O^b_{1,2}\big\|_{L^\infty_tH^3_x}
\lesssim&\|a^2_{(b,k)}\|_{L^\infty_t H^4_x}\|\g^2\|_{L^\infty_t}(\|\psi^2_k\|_{L^\infty_x}\|\p^2\|_{H^3_x}+\|\psi^2_k\|_{H^3_x}\|\p^2\|_{L^\infty_x})\\
&+\|a^2_{(b,k)}\|_{L^\infty_t H^4_x}\|\g^2\|_{L^\infty_t}\|\p^2\|_{H^3_x}\\
\lesssim&\lambda^{C_0}_q\lambda^{\sigma+3+\gamma}_{q+1}.
\end{aligned}
\end{equation}
Thanks to Proposition \ref{es-wdLpLinfty}, we have
\begin{equation}\label{MoscH33}
\begin{aligned}
\|O^b_2\|_{L^\infty_t H^3_x}
\lesssim &\|(w ^{(c)}_{q+1}, w^{(t)}_{q+1}, d^{(c)}_{q+1}, d^{(t)}_{q+1})\|_{L^\infty_{t,x}}\|(d ^{(p)}_{q+1}, d_{q+1}, w^{(p)}_{q+1}, w_{q+1})\|_{L^\infty_tH^3_x}\\
&+\|(w ^{(c)}_{q+1}, w^{(t)}_{q+1}, d^{(c)}_{q+1}, d^{(t)}_{q+1})\|_{L^\infty_tH^3_x}\|(d ^{(p)}_{q+1}, d_{q+1}, w^{(p)}_{q+1}, w_{q+1})\|_{L^\infty_{t,x}}\\
\lesssim &\lambda^{2C_0}_q\lambda^{5+\gamma}_{q+1}.
\end{aligned}
\end{equation}
Collecting the above three estimates together yields \eqref{MoscH3}.
\end{proof}
\subsection{Estimates for the Reynolds stress}
The proof of Proposition \ref{Mlin} can be applied to estimate $\Rlin-\mathcal{R}\del_t w^{(t)}_{q+1}$, we give the result  directly:
\begin{prop}
\begin{align}
&\|\Rlin-\mathcal{R}\del_t w^{(t)}_{q+1}\|_{L^1_{t,x}}\lesssim \lambda^{2C_0}_q(\lambda^{\sigma-\frac{\gamma}{4}}_{q+1}+\lambda^{-\frac{\gamma}{6}}_{q+1}+\lambda^{-\sigma}_{q+1}),\label{RlinL1}\\
&\|\Rlin-\mathcal{R}\del_t w^{(t)}_{q+1}\|_{L^\infty_tH^3_x}\lesssim \lambda^{C_0}_{q}\lambda^{5+\sigma+\frac{\gamma}{2}}_{q+1}\label{RlinH3}.
\end{align}
\end{prop}
Next, we turn to estimate the oscillation term and the temporal flow part in terms of the Reynolds stress.
\begin{prop}\label{Ros}
\begin{align}
&\|\Rosc+\mathcal{R}_a \partial_t w^{(t)}_{q+1}\|_{L^1_{t,x}}\lesssim \lambda^{C_0}_q(
\lambda^{-\frac{\gamma}{6}}_{q+1}+\lambda^{\gamma-\frac{1}{2}}_{q+1}+\lambda^{-\sigma}_{q+1}),\label{RoscL1}\\
&\|\Rosc+\mathcal{R}_a \partial_t w^{(t)}_{q+1}\|_{L^\infty_t H^3_x}\lesssim \lambda^{2C_0}_{q}\lambda^{5+\gamma}_{q+1}.\label{RoscH3}
\end{align}
\end{prop}
\begin{proof}We decompose $\Rosc+\mathcal{R}_a \partial_t w^{(t)}_{q+1}$ into two parts:
\begin{equation}\label{Rosc}
\begin{aligned}
\Rosc+\mathcal{R} \partial_t w^{(t)}_{q+1}
=&\mathcal{R}(\div(\RRR_q+w^{(p)}_{q+1}\otimes w^{(p)}_{q+1}-d^{(p)}_{q+1}\otimes d^{(p)}_{q+1})+ \partial_t w^{(t)}_{q+1}-\nabla p_v)\\
&+\mathcal{R}\div((w^{(c)}_{q+1}+w^{(t)}_{q+1})\otimes w_{q+1}+w^{(p)}_{q+1}\otimes (w^{(c)}_{q+1}+w^{(t)}_{q+1}))\\
&-(d^{(c)}_{q+1}+d^{(t)}_{q+1})\otimes d_{q+1}+d^{(p)}_{q+1}\otimes (d^{(c)}_{q+1}+d^{(t)}_{q+1}))\\
:=&O^v_1+O^v_2.
\end{aligned}
\end{equation}
By the definition of $(w^{(p)}_{q+1}, d^{(p)}_{q+1})$, we rewrite $O^v_1$ to  be
\begin{align*}
O^v_1=&\mathcal{R}\div\Big(\RRR_q+\sum_{k\in \Lambda_v}a^2_{(v,k)}\p^2\psi^2_k\g^2\bar{k}\otimes \bar{k}\\
&+\sum_{k\in \Lambda_b}a^2_{(b,k)}\p^2\psi^2_k\g^2(\bar{k}\otimes \bar{\bar{k}}-\bar{\bar{k}}\otimes \bar{k})+\partial_t w^{(t)}_{q+1}-\nabla p_v\Big)\\
&+\mathcal{R}\div\Big(\sum_{ k\neq k'\in \Lambda_v}a_{(v,k)}a_{(v,k')}\p\pd\psi_k\psi_{k'}\g^2\bar{k}\otimes \bar{k'}\\
&+\sum_{ k\neq k'\in \Lambda_b}{a_{(b,k)}a_{(b,k')}}\p\pd\psi_k\psi_{k'}\g^2\bar{k}\otimes \bar{k'}\\
&+\sum_{ k\in\Lambda_v, k'\in \Lambda_b}{a_{(v,k)}a_{(b,k')}}\p\pd\psi_k\psi_{k'}\g^2\bar{k}\otimes \bar{k'}\\
&-\sum_{k\neq k'\in \Lambda_b}{a_{(b,k)}a_{(b,k')}}\p\pd\psi_k\psi_{k'}\g^2\bar{\bar{k}}\otimes \bar{\bar{k'}}\Big)\\
:=&O^v_{1,1}+O^v_{1,2}.
\end{align*}
Similar to the argument  as in the proof of  \eqref{Moscd}, we have
$\|O^v_{1,2}\|_{L^1_{t,x}}\lesssim {\lambda^{24}_q}\lambda^{-\frac{\gamma}{6}}_{\lambda_{q+1}}.$

For $O^v_{1,1}$, note the fact that $k\perp \bar{k}\perp  \bar{\bar{k}}$, we decompose it as follows:
\begin{align*}
O^v_{1,1}=&\mathcal{R}\div (\RRR_q+\sum_{k\in \Lambda_v}a^2_{(v,k)}\bar{k}\otimes \bar{k}+\sum_{ k\in \Lambda_b}a^2_{(b,k)}(\bar{k}\otimes \bar{k}-\bar{\bar{k}}\otimes \bar{\bar{k}}))\\
&+\mathcal{R}\div \sum_{ k\in \Lambda_v}a^2_{(v,k)}\PP(\psi^2_k)
\p^2\g^2\bar{k}\otimes \bar{k}\\
&+\mathcal{R}\div \sum_{k\in \Lambda_v}a^2_{(v,k)}\PP\big(\p^2)\g^2\bar{k}\otimes \bar{k}\\
&+\mathcal{R}\div \sum_{k\in \Lambda_b}a^2_{(b,k)}\PP(\psi^2_k)
\p^2\g^2(\bar{k}\otimes \bar{k}-\bar{\bar{k}}\otimes \bar{\bar{k}})\\
&+\mathcal{R}\div \sum_{k\in \Lambda_b}a^2_{(b,k)}\PP(\p^2)
\g^2(\bar{k}\otimes \bar{k}-\bar{\bar{k}}\otimes \bar{\bar{k}})\\
&+\mathcal{R}(\sum_{k\in \Lambda_v}\div (a^2_{(v,k)}\bar{k}\otimes \bar{k})\PP(\g^2)+\partial_t w^{(t)}_{q+1}\\
&+\sum_{ k\in \Lambda_b}\div(a^2_{(b,k)}\bar{k}\otimes \bar{k}-\bar{\bar{k}}\otimes \bar{\bar{k}})\PP(\g^2)-\nabla p_v).
\end{align*}
{Thanks to  \eqref{wt} and \eqref{pv}, it is easy to deduce that}
\begin{align*}
&\sum_{k\in \Lambda_v}\div (a^2_{v,k}\bar{k}\otimes \bar{k})\PP\big(\g^2\big)+\partial_t w^{(t)}_{q+1}+\sum_{ k\in \Lambda_b}\div(a^2_{b,k}\bar{k}\otimes \bar{k}-\bar{\bar{k}}\otimes \bar{\bar{k}})\PP\big(\g^2\big)-\nabla p_v\\
=&\sum_{k\in \Lambda_v}\frac{\div\div}{\Delta}a^2_{(v,k)}(\bar{k}\otimes\bar{k})\PP(\g^2)+\sum_{k\in \Lambda_b}\frac{\div\div}{\Delta}a^2_{(b,k)}(\bar{k}\otimes\bar{k}-\bar{\bar{k}}\otimes\bar{\bar{k}})\PP(\g^2)-\nabla p_v\\
=&-\eta^2_q\nabla\rho_v.
\end{align*}
By Lemma \ref{first S} and the definition of $R_v$,  the first three terms can be simplified as
\begin{align*}
&\div(\RRR_q+\sum_{k\in \Lambda_v}\eta^2_q\rho_va^2_{v,k}({\rm Id}-\tfrac{R_v}{\rho_v})\bar{k}\otimes \bar{k}+\sum_{k\in \Lambda_b}\eta^2_q\rho_va^2_{b,k}(-\tfrac{\MMM_q}{\rho_v})(\bar{k}\otimes \bar{k}-\bar{\bar{k}}\otimes \bar{\bar{k}}))\\
=&\eta^2_q\div(\RRR_q+\rho_v {\rm Id}-R_v+\sum_{k\in \Lambda_b}\eta^2_q\rho_va^2_{b,k}(-\tfrac{\MMM_q}{\rho_v})(\bar{k}\otimes \bar{k}-\bar{\bar{k}}\otimes \bar{\bar{k}}))=\eta^2_q\nabla\rho_v.
\end{align*}
Collecting the above two equalities, we have
\begin{align*}
&\div\big(\RRR_q+\sum_{k\in \Lambda_u}a^2_{(v,k)}\bar{k}\otimes \bar{k}+\sum_{ k\in \Lambda_b}a^2_{(b,k)}(\bar{k}\otimes \bar{k}-\bar{\bar{k}}\otimes \bar{\bar{k}})\big)+\sum_{k\in \Lambda_v}\div (a^2_{(v,k)}\bar{k}\otimes \bar{k})\PP(\g^2)\\
&+\partial_t w^{(t)}_{q+1}+\sum_{ k\in \Lambda_b}\div(a^2_{(b,k)}\bar{k}\otimes \bar{k}-\bar{\bar{k}}\otimes \bar{\bar{k}})\PP(\g^2)-\nabla p_v=0.
\end{align*}
Thus, utilizing similar methods of dealing with \eqref{M2} and \eqref{M2}, one infers that
\begin{align*}
\|O^v_{1,1}\|_{L^1_{t,x}}\lesssim&\|\mathcal{R}\div (a^2_{(v,k)}\bar{k}\otimes \bar{k})\PP(\psi^2_k)
\p^2\g^2\|_{L^1_{t,x}}\\
&+\|\mathcal{R}(\div a^2_{(v,k)}\bar{k}\otimes \bar{k} )\PP\big(\p^2)\g^2\|_{L^1_{t,x}}\\
&+\|\mathcal{R}\div (a^2_{(b,k)}\bar{k}\otimes \bar{k}-\bar{\bar{k}}\otimes \bar{\bar{k}})\PP(\psi^2_k)
\p^2\g^2\|_{L^1_{t,x}}\\
&+\|\mathcal{R}\div (a^2_{(b,k)}\bar{k}\otimes \bar{k}-\bar{\bar{k}}\otimes \bar{\bar{k}})\PP(
\p^2)\g^2\|_{L^1_{t,x}}
\lesssim\lambda^{C_0}_q\lambda^{-\frac{1}{2}}_{q+1}.
\end{align*}
Imitating the method of \eqref{dlinL1}, it is easy to get that
\begin{equation}\label{dlinL1}
\begin{aligned}
\|O^v_2\|_{L^1_{t,x}}
\lesssim &(\|(w ^{(c)}_{q+1}, w^{(t)}_{q+1}, d^{(c)}_{q+1}, d^{(t)}_{q+1})\|_{L^2_{t,x}}\|(d ^{(p)}_{q+1}, d_{q+1}, w^{(p)}_{q+1}, w_{q+1})\|_{L^2_{t,x}}\\
\lesssim &\lambda^{C_0}_q\lambda^{\gamma-\frac{1}{2}}_{q+1}
+\lambda^{C_0}_q\lambda^{-\sigma}_{q+1}.
\end{aligned}
\end{equation}
In the same way as deriving \eqref{MoscH31}--\eqref{MoscH33}, we infer \eqref{RoscH3}. Therefore, we finish the proof of Proposition \ref{Ros}.
\end{proof}
\section{Verification of iterative estimates at $q+1$ level}
In this section, we collect estimates of glued parts $(\vv_q, \bb_q)$, the perturbation $(w_{q+1}, d_{q+1})$, the Reynolds and magnetic stresses $(\RRR_{q+1}, \MMM_{q+1})$ to show they satisfies \eqref{e:vq-LpLinftyC1-}--\eqref{e:initial} under conditions \eqref{gamma}-\eqref{10}.

Firstly,  one shows by Proposition \ref{estimate-vvq} and Proposition \ref{es-wdL2} that
\begin{align*}
&\|(v_{q+1}, b_{q+1})\|_{L^2_{t,x}}+\|(v_{q+1}, b_{q+1})\|_{L^p_tL^\infty_x}+\|(v_{q+1}, b_{q+1})\|_{L^1_tC^{1-\e}_x}\\
\le& \sum_{i=1}^q\delta^{1/2}_i+C\lambda^{-\frac{3}{2}}_q+\delta^{\frac{1}{2}}_{q+1}\lambda^{-\frac{\alpha}{2}}_q+\lambda^{C_0}_q\lambda^{\gamma-\frac{1}{2}}_{q+1}+\lambda^{C_0}_q\lambda^{1+\frac{\gamma}{2}-\frac{2}{p}}_{q+1}+2\lambda^{C_0}_q\lambda^{-\sigma}_{q+1}
+\lambda^{C_0}_q\lambda^{\frac{\gamma}{2}-\e}_{q+1}\\
\le&\sum_{i=1}^q\delta^{1/2}_i+\delta^{1/2}_{q+1},
\end{align*}
where the last inequality is valid with  parameters satisfying condition \eqref{gamma} and large enough~$a$.  Next, combining Proposition \ref{estimate-vvq} with \eqref{dH3}, we have
\begin{align*}
&\|(v_{q+1}, b_{q+1})\|_{L^\infty_t H^3_x}
\lesssim\lambda^6_q+{\lambda^{C_0}_q}\lambda^{4+\frac{\gamma}{2}}_{q+1}\le \lambda^{6}_{q+1},
\end{align*}
where the last inequality holds thanks to $\gamma<1$ and $b>\tfrac{C_0}{2-\tfrac{\gamma}{2}}$ in \eqref{gamma}. Collecting \eqref{MlinH3}, \eqref{MoscH3}, \eqref{RlinH3} and \eqref{RoscH3}, one obtains that
\begin{align*}
&\|(\RR_{q+1}, \MM_{q+1})\|_{L^\infty_t H^3_x}
\lesssim{\lambda^{2C_0}_q}\lambda^{5+\sigma+\frac{\gamma}{2}}_{q+1}+{\lambda^{2C_0}_q}\lambda^{5+{\gamma}}_{q+1}\le \lambda^{6}_{q+1},
\end{align*}
where the last inequality is valid because $a$ is large enough, $\sigma+\frac{\gamma}{2}<1$ and $b>\tfrac{2C_0}{1-\sigma-\tfrac{\gamma}{2}}$.

Finally, with aid of \eqref{MlinL1}, \eqref{MoscL1}, \eqref{RlinL1} and \eqref{RoscL1}, thanks to $$\sigma<\max\{\tfrac{\gamma}{6}, \tfrac{1}{2}-\gamma, \tfrac{\gamma}{2}\}, \alpha<\frac{\sigma}{18}, b\beta<\frac{\sigma}{6},  b>\frac{6C_0}{\sigma},$$
one gets that
\begin{align*}
&\|(\RR_{q+1}, \MM_{q+1})\|_{L^1_{t,x}}
\lesssim\lambda^{2C_0}_q(\lambda^{-\frac{\gamma}{6}}_{q+1}+\lambda^{\gamma-\frac{1}{2}}_{q+1}+\lambda^{-\sigma}_{q+1}+\lambda^{\sigma-\frac{\gamma}{4}}_{q+1} )\le\delta_{q+2}\lambda^{-6\alpha}_{q+1}.
\end{align*}
Therefore, we prove that estimates \eqref{e:vq-LpLinftyC1-}--\eqref{e:RR_q-L1} hold with $q$ replaced by $q+1$  and \eqref{e:velocity-diff} is valid. Thanks to \eqref{e:vv_q-initial}, \eqref{dwt} and $10\tau_{q+1}<\tau_q$, it is easy to verify that
\begin{align*}
&(v_{q+1}, b_{q+1})\equiv(\vv_q, \bb_q)\equiv(v^{(1)}, b^{(1)}) \,\text {on}\,[0, 2T+\tfrac{7\tau_q}{6}]\times\TTT^3 \supset[0, 2T+2\tau_{q+1}]\times\TTT^3,\\
&(v_{q+1}, b_{q+1})\equiv(\vv_q, \bb_q)\equiv(v^{(2)}, b^{(2)}) \,\text {on}\,[3T-\tfrac{\tau_q}{6}, 1]\times\TTT^3 \supset[3T-\tfrac{\tau_{q+1}}{3}, 1]\times\TTT^3.
\end{align*}
Therefore, we complete the proof of Proposition \ref{p:main-prop}.
\section*{Acknowledgement}
This work is supported by the National Natural Science Foundation of China under grant No. 11871087 and Y. Nie is partially supported by CPSF (2021M700510).
\section{Appendix}
\subsection{Some technical tools}
\begin{lem}[\cite{1Cheskidov, 13Nonuniqueness}]\label{Holder}Assume that $1\le p\le \infty$ and $\lambda$ is a positive integer. Let $f,g:\TTT^d\mapsto \RR$ be smooth functions. Then we have
\[\big|\|fg(\lambda\cdot)\|_{L^p}-\|f\|_{L^p}\|g\|_{L^p}\big|\lesssim \lambda^{-\frac{1}{p}}\|f\|_{C^1}\|g\|_{L^p}.\]
\end{lem}
\begin{lem}[\cite{2Beekie,13Nonuniqueness}]\label{holder2}
Fixed $\kappa>\lambda\geq1$ and $p\in (1,2]$. Suppose that there exists an integer $M$ with $\kappa^{M-2}>\lambda^{M}$. Let $f$ be a $\mathbb{T}^3$-periodic function so that there exists a constant $C_f$ such that
$$\|D^jf\|_{L^{\infty}}\leq C_f\lambda^j,~~~ j\in[0,M].$$
Assume furthermore that $\int_{\mathbb{T}^3}f(x)\mathbb{P}_{\geq \kappa}g(x)dx=0$ and $g$ is a $(\mathbb{T}/\kappa)^3$-periodic function. Then, it holds thst
$$\||\nabla|^{-1}(f\mathbb{P}_{\geq \kappa}g)\|_{L^p}\lesssim  C_f\frac{\|g\|_{L^p}}{\kappa},$$
where the implicit constant is universal.
\end{lem}
\subsection{Symmetric and skew-symmetric inverse divergence}\label{inversedivergence}
\label{ss:inverse-div}
We recall  \cite{2Beekie} the following pseudodifferential operator of order $-1$,
\begin{align}\label{Ru}
    &{\mathcal{R} v}^{kl} = \partial_k\Delta^{-1}v^l+ \partial_l\Delta^{-1}v^k-\frac{1}{2}
    (\delta_{kl}+\partial_k\partial_l\Delta^{-1})\div\Delta^{-1}v,\\
       &\mathcal{R}_{a} u = \Delta^{-1} (\nabla u + \nabla u^\TT ) - \frac12\Delta^{-2} \nabla^2 \nabla\cdot u  - \frac12 \Delta^{-1} (\nabla\cdot u )I_{3\times 3},
\end{align}
$\mathcal{R}$ is a matrix-valued right inverse of the divergence operator for mean-free vector fields $v$, in the sense that
$\div\mathcal R v= v $. While $\mathcal{R}_{a}$ is also a matrix-valued right inverse of the divergence operator for divergence free vector fields $v$, in the sense that $\div\mathcal{R}_{a} u= u $. $\mathcal Rv$ is traceless and symmetric, and $\mathcal{R}_{a}u$ is antisymmetric.

\addcontentsline{toc}{section}{\refname}
\bibliography{reference}
\bibliographystyle{plain}

\end{document}